\documentclass[sn-mathphys,Numbered]{sn-jnl}

\usepackage{graphicx}%
\usepackage{multirow}%
\usepackage{amsmath,amssymb,amsfonts}%
\usepackage{amsthm}%
\usepackage{mathrsfs}%
\usepackage[title]{appendix}%
\usepackage{xcolor}%
\usepackage{textcomp}%
\usepackage{manyfoot}%
\usepackage{booktabs}%
\usepackage{algorithm}%
\usepackage{algorithmicx}%
\usepackage{algpseudocode}%
\usepackage{listings}%

\theoremstyle{thmstyleone}%
\newtheorem{theorem}{Theorem}
\newtheorem{proposition}[theorem]{Proposition}%

\theoremstyle{thmstyletwo}%
\newtheorem{example}{Example}%
\newtheorem{remark}{Remark}%
\newtheorem{lemma}{Lemma}%
\theoremstyle{thmstylethree}%
\newtheorem{definition}{Definition}%

\begin{document}

\title[Article Title]{Solving maximally comonotone inclusion problems via an implicit Newton-like inertial dynamical system and its discretization}


\author[1]{\fnm{Zengzhen} \sur{Tan}}\email{zengzhentan@foxmail.com}

\author[2]{\fnm{Rong} \sur{Hu}}\email{ronghumath@aliyun.com}

\author*[1]{\fnm{Yaping} \sur{Fang}}\email{ypfang@scu.edu.cn}

\affil[1]{\orgdiv{Department of Mathematics}, \orgname{Sichuan University}, \orgaddress{\city{Chengdu}, \postcode{610065}, \state{Sichuan}, \country{China}}}

\affil[2]{\orgdiv{Department of Applied Mathematics}, \orgname{Chengdu University of Information Technology}, \orgaddress{\city{Chengdu}, \postcode{610225}, \state{Sichuan}, \country{China}}}

\abstract{This paper deals with an implicit Newton-like inertial dynamical system governed by a maximally comonotone inclusion problem in a Hilbert space. Under suitable conditions, we establish  not only  pointwise estimates  and integral estimates for the velocity and the value of  the associated Yosida regularization operator along the trajectory of the  system, but also the weak convergence of the trajectory to a zero of the maximally comonotone operator. Moreover, a new inertial algorithm is developed via a time discretization of the proposed system. Our analysis reveals that the resulting discrete algorithm  exhibits fast convergence properties  matching the ones of the continuous time counterpart. Finally,  the theoretical results are illustrated by  numerical experiments.}



\maketitle

\section{Introduction}\label{sec1}

Let $\mathcal{H}$ be a real Hilbert space and $\mathcal{A}:\mathcal{H}\rightarrow 2^{\mathcal{H}}$ be a set-valued operator such that $\mbox{zer}  \mathcal{A}:= {\mathcal{A}}^{-1}(0)\neq \emptyset $. Designing continuous time dynamical systems and discrete algorithms with fast convergence properties to solve the inclusion problem:
\begin{equation}\label{problem}
\mbox{ Find } \quad x^*\in \mathcal{H}\quad \mbox{such that} \quad 0\in \mathcal{A}(x^*)
\end{equation}
is of great significance in many domains, including optimization, equilibrium theory, game theory and economics, partial differential equations, statistics and other disciplines. When $\mathcal{A}=\partial f$, the inclusion problem $(\ref{problem})$ becomes the  optimization problem
\begin{equation}\label{opp}
\min_{x\in \mathcal{H}}f(x),	
\end{equation}
where $f: \mathcal{H}\rightarrow \mathbb{R}\bigcup \{+\infty\}$ is a proper, convex and lower semi-continuous function and $\partial f$ denotes its sub-differential. 

There is a long history of using dynamical systems to solve the optimization problem $(\ref{opp})$. 
Let's start with the seminal work of the heavy ball with friction system proposed by Polyak \cite{Polyak}
\begin{eqnarray}\label{HVB}
		\ddot{x}(t)+\gamma \dot{x}(t)+\nabla f(x(t))=0, 
\end{eqnarray}
where $\gamma>0$ is  a fixed damping coefficient.  The explicit discretization scheme of \eqref{HVB} leads to the heavy ball method for the problem $(\ref{opp})$. The next major development in dynamical system approaches is due to Su et al. \cite{Su}, who introduced the inertial system of the form
\begin{equation}\label{Su}
\ddot{x}(t)+\frac{\alpha}{t}\dot{x}(t)+\bigtriangledown f(x(t))=0	
\end{equation}
for understanding  Nesterov's accelerated gradient method  \cite{Nesterov} for the problem $(\ref{opp})$, where $\alpha>0$.  Su et al. \cite{Su}  showed that the case $\alpha=3$  is corresponding to the continuous time limit of   Nesterov's accelerated gradient method  \cite{Nesterov}  and proved the fast convergence property $f(x(t))-\min_{\mathcal{H}}f=O(t^{-2})$ for \eqref{Su} with  $\alpha\geq 3$. For $\alpha>3$, Attouch et al. \cite{AttouchC} and May \cite{May} proved the weak convergence of the trajectory generated by $(\ref{Su})$ to a solution of \eqref{opp} and the improved convergence rate $f(x(t))-\min_{\mathcal{H}}f=o(t^{-2})$. Attouch et al. \cite{AttouchPR} further proposed the  inertial system with Hessian damping term 
\begin{equation}\label{Hes}
\ddot{x}(t)+\frac{\alpha}{t}\dot{x}(t)+\beta{\bigtriangledown}^2 f(x(t))\dot{x}(t)+{\bigtriangledown}f(x(t))=0,	
\end{equation}
where $\beta\geq 0$ and $f$ is a convex $\mathcal{C}^{2}$ function, which retains  convergence properties of \eqref{Su}.  In contrast to  \eqref{Su},  the addition of the Hessian driven damping term ${\bigtriangledown}^2 f(x(t))\dot{x}(t)$  makes the system \eqref{Hes} linked to the Newton method and to  enjoy additional favorable properties: Fast convergence of the gradients to zero and attenuation of the oscillations. See e.g. \cite{AttouchPR,Shib,AttouchF}.

Although the system  \eqref{Su} models  Nesterov's accelerated gradient method \cite{Nesterov}  from the continuous time limit perspective, it does not give  inertial gradient-type methods via a natural explicit/implicit discretization scheme.  To overcome this drawback, Alecsa et al. \cite{Alecsa} introduced the inertial dynamical system with an implicit Hessian damping term
\begin{equation}\label{im}
\ddot{x}(t)+\frac{\alpha}{t}\dot{x}(t)+\bigtriangledown f\left(x(t)+(\gamma+\frac{\beta}{t})\dot{x}(t)\right)=0,	
\end{equation}
where $\alpha>0, \beta\in \mathbb{R}$, and $\gamma\geq 0$. The term $\bigtriangledown f\left(x(t)+(\gamma+\frac{\beta}{t})\dot{x}(t)\right)$ in $(\ref{im})$ includes implicitly a Hessian damping term because by using the Taylor expansion,
 $$\bigtriangledown f\left(x(t)+(\gamma+\frac{\beta}{t}\right)\dot{x}(t))\approx \bigtriangledown f(x(t))+(\gamma+\frac{\beta}{t}){\bigtriangledown}^2f(x(t))\dot{x}(t).$$ 
It was shown in \cite{Alecsa}  that  \eqref{im} enjoys convergence properties similar to  \eqref{Hes}. In contrast to \eqref{Hes}, the system \eqref{im} owns the following advantages: (i) Its explicit discretization leads to inertial gradient-type methods. In particular, Nesterov's accelerated gradient method \cite{Nesterov} can be obtained via a natural discretization of  \eqref{im}. (ii) In numerical experiments, the trajectories generated by \eqref{im} have a better convergence behavior than the ones generated by \eqref{Hes}. It is worth mentioning that the work on inertial dynamical systems with implicit Hessian damping terms can date back to  Muehlebach and Jordan \cite{Muehlebach} who proposed 
\begin{equation}\label{im-MJ}
\ddot{x}(t)+\frac{2}{\sqrt{\kappa}+1}\dot{x}(t)+\frac{1}{L}\bigtriangledown f\left(x(t)+\frac{\sqrt{\kappa}-1}{\sqrt{\kappa}+1}\dot{x}(t)\right)=0	
\end{equation}
for minimizing  a  $\kappa$-strongly convex function  $f$ with $\nabla f$ being $L$-Lipschitz continuous. They showed that the accelerated gradient method
presented in \cite[p. 81]{Nesterov2004} results from a semi-implicit discretization of \eqref{im-MJ}.  Inspired by  the work of Alecsa et al. \cite{Alecsa},  Attouch et al. \cite{Attouch}  considered a perturbed version of  $(\ref{im})$, and
Alecsa et al. \cite{AlecsaL} and L\'aszl\'o \cite{Laszlo} investigated  Tikhonov regularized variants of $(\ref{im})$. For more results on inertial dynamical systems with implicit Hessian damping terms for the problem $(\ref{opp})$, we refer the reader to \cite{AttouchB,AlecsaR,AttouchN,Boffi} and the references therein. 

In past decades, dynamical system approaches were also extended to  solve the inclusion problem $(\ref{problem})$. Given a maximally monotone operator $\mathcal{A}$,   Attouch and Peypouquet \cite{AttouchP} considered the inertial dynamical systems with vanishing damping
 \begin{equation}\label{INS-ma}
\ddot{x}(t)+\frac{\alpha}{t}\dot{x}(t)+\mathcal{A}_{\lambda(t)}(x(t))=0,
\end{equation}
where  $\mathcal{A}_{\lambda(t)}$ is the Yosida regularization of  $\mathcal{A}$ with  a positive time-dependent  regularization parameter $\lambda(t)$.
Under suitable conditions, they proved that the trajectory $x(t)$ generated by \eqref{INS-ma} converges weakly to a zero of $\mathcal{A}$ and $\|\dot{x}(t)\|=O(1/t)$. By introducing an additional  Newton-like correction term to \eqref{INS-ma},  Attouch and L\'aszl\'o \cite{AttouchSC} proposed  the Newton-like inertial dynamical system
 \begin{equation}\label{sys6}
\ddot{x}(t)+\frac{\alpha}{t}\dot{x}(t)+b\frac{d}{dt}(\mathcal{A}_{\lambda(t)}(x(t)))+\mathcal{A}_{\lambda(t)}(x(t))=0,
\end{equation}
where the term $\frac{d}{dt}(\mathcal{A}_{\lambda(t)}(x(t)))$ is corresponding to the Hessian driven term ${\bigtriangledown}^2 f(x(t))\dot{x}(t)$ in \eqref{Hes}. It was shown in \cite[Theorem 2]{AttouchSC} that under suitable conditions, the trajectory of \eqref{sys6} converges weakly, along with the rate $\|\dot{x}(t)\|=o(1/t)$, to a zero of the maximal monotone operator  $\mathcal{A}$. By combining implicit Hessian damping with implicit Newton-like correction term, Adly et al. \cite{Adly} proposed the following inertial dynamical system 
\begin{equation}\label{B}
\ddot{x}(t)+\alpha\dot{x}(t)+\bigtriangledown f\left(x(t)+\beta_{f}\dot{x}(t)\right)+B\left(x(t)+\beta_b\dot{x}(t)\right)=0,
\end{equation}
where  $f: \mathcal{H}\rightarrow  \mathbb{R}$ is a continuously differentiable and convex function,  $B: \mathcal{H}\rightarrow \mathcal{H}$ is a monotone and  cocoercive operator, $\beta_{f}$ and $\beta_b$ are two nonnegative constants. Under certain conditions, they proved  the weak convergence $x(t)$ of the trajectory generated by \eqref{B} to  a solution of the structured monotone equation $\bigtriangledown f(x)+B(x)=0$.

Meanwhile, various new discrete algorithms were developed for solving the inclusion problem $\eqref{problem}$ with $\mathcal{A}$ being maximally monotone. See e.g. \cite{AttouchA,AttouchP,Attouchsc,Kim,Mainge}.  Some  discrete algorithms beyond the monotonicity of $\mathcal{A}$ were also proposed  for the inclusion problem $\eqref{problem}$ and the (weak) convergence of the iterative sequences generated by the algorithms were established. See e.g. \cite{CP,Iusem,Penn,BartzD,Kohlenbach,Cai,Liu,BartzC,BartzP,Quoc,Evens}.  Compared to the results on the convergence of the iterative sequences,  in  the literature it was limited to address the convergence rates of the algorithms.

Recently, Tan et al.  \cite{Tan} proposed a Newton-like inertial dynamical system for solving the inclusion problem \eqref{problem} with  $\mathcal{A}$ being a maximally  comonotone  operator (see Definition \ref{comonotone}), which is formulated as
\begin{equation}\label{ds}
\ddot{x}(t)+\frac{\alpha}{t}\dot{x}(t)+\frac{\beta}{t}\mathcal{A}_{\eta}(x(t))+\frac{d}{dt}\mathcal{A}_{\eta}(x(t))\dot{x}(t)=0,
\end{equation}	
where $\alpha>0$ and $\beta>0$.
To the best of our knowledge, this is the first dynamical system in the literature for the  maximally  comonotone  inclusion problem. Under certain conditions, they proved the weak convergence of the trajectory $x(t)$ generated by \eqref{ds} to a zero of the maximally  comonotone operator $\mathcal{A}$, along with   $\|\dot{x}(t)\|=o(1/t)$ and  $\|\mathcal{A}_{\eta}(x(t))\|=o(1/t)$. Further,  Tan et al.  \cite{Tan} also developed  via time discretization of \eqref{ds}  the following  inertial algorithm
\begin{equation}\label{Al-cns}
\left\{
\begin{array}{rcl}
y_{k}&=&x_k+(1-\frac{\alpha}{k})(x_k-x_{k-1})+(1-\frac{\beta}{k})(y_{k-1}-x_k),\\
x_{k+1}&=&(1-\frac{1}{\eta+1})y_k+ \frac{1}{\eta+1}J_{\eta+1}^{\mathcal{A}}y_k,
\end{array}
\right.
\end{equation}
where $J_{\eta+1}^{\mathcal{A}}$ is the resolvent operator of $\mathcal{A}$ with index $\eta+1$, and proved that under certain conditions the algorithm \eqref{Al-cns} exhibits convergence properties matching the ones of \eqref{ds}.  

Inspired by \cite{Alecsa,Adly,Tan}, in this paper, we propose the following implicit Newton-like inertial dynamical system for  finding a zero of the maximally  comonotone operator $\mathcal{A}$:
\begin{equation}\label{DS}
\ddot{x}(t)+\frac{\alpha}{t}\dot{x}(t)+\frac{\beta}{t}\mathcal{A}_{\eta}\left(x(t)+\frac{t}{\gamma}\dot{x}(t)\right)=0,
\end{equation}	 
where $t\ge t_0>0$ and $\alpha, \beta, \eta,\gamma >0$. The third term in the left side of \eqref{DS} includes implicitly a  Newton-like correction term. Indeed, by using the Taylor expansion we get
$$\mathcal{A}_{\eta}\left(x(t))+\frac{t}{\gamma}\dot{x}(t)\right)\approx \mathcal{A}_{\eta}(x(t))+\frac{t}{\gamma}\frac{d}{dt}\mathcal{A}_{\eta}(x(t))\dot{x}(t).$$
In this sense, the system \eqref{DS} is strongly related to \eqref{ds}.

In this paper we also consider the following  discretization scheme of \eqref{DS} with the constant stepsize $1$: 
$$x_{n+1}-2x_n+x_{n-1}+\frac{\alpha}{n}(x_n-x_{n-1})+\frac{\beta}{n}\mathcal{A}_{\eta}\left(x_n+\frac{n}{\gamma}(x_n-x_{n-1})\right)=0,$$
which can be rewritten as
\begin{equation}\label{al}
 x_{n+1}=x_n+(1-\frac{\alpha}{n})(x_n-x_{n-1})-\frac{\beta}{n}\mathcal{A}_{\eta}\left(x_n+\frac{n}{\gamma}(x_n-x_{n-1})\right),	
\end{equation}
or equivalently,
\begin{equation}\label{ale}
\left\{
\begin{array}{rcl}
y_{n}&=&x_n+(1-\frac{\alpha}{n})(x_n-x_{n-1}),\\
z_n&=&x_n+\frac{n}{\gamma}(x_{n}-x_{n-1}),\\
x_{n+1}&=& y_n-\frac{\beta}{n}\mathcal{A}_{\eta}(z_n),
\end{array}
\right.
\end{equation}
where  $x_0$, $x_{-1}\in \mathcal{H}$. Notice that the algorithm \eqref{ale} is quite different from \eqref{Al-cns} although its continuous time counterpart  \eqref{DS} is very similar to \eqref{ds}. We will show that the system \eqref{DS} enjoys fast convergence properties similar to  \eqref{ds}  and that \eqref{ale} owns  similar theoretical convergence properties and but better performance in numerical experiments over \eqref{Al-cns}.

Given a maximally comonotone operator $\mathcal{A}$,  let $x(t)$ be  the trajectory of  \eqref{DS},  and let $x_n$ and $z_n$ be the sequences generated by \eqref{ale}. Our main results are summarized as follows:  Under certain conditions, we have
\begin{itemize}
\item[]$\bullet$ Integral estimates:
\begin{enumerate}
\item[(i)] $\int_{t_0}^{+\infty}t\|\dot{x}(t)\|^2dt<+\infty$ and $\int_{t_0}^{+\infty} t\|\mathcal{A}_{\eta}(x(t)+\frac{t}{\gamma}\dot{x}(t))\|^2 dt <+\infty$.
\item[(ii)] $\sum_{n=1}^{+\infty} n\|x_n-x_{n-1}\|^2<+\infty$,  $\sum_{n=1}^{+\infty} n\|\mathcal{A}_{\eta}(z_{n})\|^2<+\infty$ and 
$\sum_{n=1}^{+\infty} n^2\|\mathcal{A}_{\eta}(z_{n})-\mathcal{A}_{\eta}(z_{n+1})\|^2<+\infty$.
\end{enumerate}

\item[]$\bullet$ Point estimates:
\begin{enumerate}
\item[(i)] $\|\dot{x}(t)\|=o(\frac{1}{t})$ and $\|\mathcal{A}_{\eta}(x(t)+\frac{t}{\gamma}\dot{x}(t))\|=o(\frac{1}{t})$ as $t\rightarrow +\infty$.
\item[(ii)]$\|x_n-x_{n-1}\|=o(\frac{1}{n})$ and $\|\mathcal{A}_{\eta}(z_n)\|=o(\frac{1}{n})$ as $n\rightarrow +\infty$.
\end{enumerate}

\item[]$\bullet$ Weak convergence:
\begin{enumerate}
\item[(i)]  The trajectory $x(t)$ converges weakly, as $t\rightarrow +\infty$, to an element of $\mbox{zer} \mathcal{A}$.
\item [(ii)] The sequence $\{x_n\}$ converges weakly, as $n\rightarrow +\infty$, to an element of $\mbox{zer} \mathcal{A}$.
\end{enumerate}
\end{itemize}

The article is structured as follows: Section \ref{sec2} gives some basic definitions and results which will be  used in our convergence analysis.  In Section \ref{sec3}, we first give  a short discussion concerning the existence and uniqueness of the trajectories generated by \eqref{DS}, then we establish the asymptotical analysis of \eqref{DS}.   Section \ref{sec4} is devoted to the  convergence analysis of the algorithm \eqref{al}. Section \ref{sec5} presents some simulation examples that illustrate the theoretical results on \eqref{DS} and \eqref{al}. In Section \ref{sec6}, we conclude our paper.

\section{Notations and preliminaries}\label{sec2}

Throughout, we will employ standard notations adopted in \cite{BauschkeC}:  $\mathcal{H}$ is a real Hilbert space with the inner product $\langle \cdot, \cdot \rangle$ and the associated norm $\|\cdot\|$. The sets $\mathbb{N}$, $\mathbb{R}$, $\mathbb{R}_{+}$ and $\mathbb{R}_{++}$ denote the set of nonnegative integers, the set of real number, the set of nonnegative real numbers, and the set of positive real numbers, respectively. The notation $\mathcal{A}:\mathcal{H}\rightarrow 2^{\mathcal{H}}$ indicates that $\mathcal{A}$ is a set-valued operator on $\mathcal{H}$ and the notation $\mathcal{A}:\mathcal{H}\rightarrow \mathcal{H}$ indicates that $\mathcal{A}$ is a single-valued operator on $\mathcal{H}$. Given $\mathcal{A}:\mathcal{H}\rightarrow 2^{\mathcal{H}}$,  $\text{ dom } \mathcal{A}:=\{x\in \mathcal{H}|\mathcal{A}x\neq \emptyset\}$ denotes  its domain, $\mbox{ gra }\mathcal{A}:=\{(x,u)\in \mathcal{H}\times \mathcal{H}|u\in \mathcal{A}x\}$ denotes its graph, and $\mbox{ zer } \mathcal{A}:=\{x\in \mathcal{H}|0\in \mathcal{A}x\}$ denotes the set of its zeros. The resolvent operator of $ \mathcal{A}$  with index $\eta>0$ is defined  by $J_{\eta}^{\mathcal{A}}:=(Id+\eta \mathcal{A})^{-1}$ and  the Yosida regularization of $ \mathcal{A}$ with parameter $\eta>0$ is defined by $\mathcal{A}_{\eta}:=\frac{1}{\eta}(Id-J_{\eta}^{\mathcal{A}})$.

\begin{definition} \cite{BauschkeC}\label{cocoercive}
 Let $T:\mathcal{H}\rightarrow \mathcal{H}$,  $\theta\in (0,1)$, and  $\beta\in \mathbb{R}_{++}$. We say that
 \begin{enumerate}
 \item[(i)] $T$ is nonexpansive if $\|Tx-Ty\|\leq \|x-y\|,\quad \forall x,y\in \mathcal{H}$.
 \item[(ii)]$T$ is $\theta-$averaged if there exists a nonexpansive operator $N:\mathcal{H}\rightarrow \mathcal{H}$ such that $T=(1-\theta)Id+\theta N$, i.e., 
 $$(1-\theta)\|(Id-T)x-(Id-T)y\|^2\leq \theta(\|x-y\|^2-\|Tx-Ty\|^2),\quad\forall x, y\in \mathcal{H}.$$
 \item[(iii)]$T$ is $\beta-$cocoercive if 
 $$\langle x-y,Tx-Ty\rangle \geq \beta \|Tx-Ty\|^2,\quad \forall x, y\in \mathcal{H}.$$	
 \end{enumerate}
\end{definition}
 
\begin{definition}  \cite[Definition 2.3] {BauschkeMW} \label{comonotone}
 Let $\mathcal{A}:\mathcal{H}\rightarrow 2^{\mathcal{H}}$ and $\rho\in \mathbb{R}$. We say that
  \begin{enumerate}
 \item[(i)] $\mathcal{A}$ is $\rho-$monotone if 
 $$\langle x-y,u-v\rangle \geq \rho \|x-y\|^2,\quad \forall (x,u), (y,v)\in \text{ gra }\mathcal{A}.$$

\item[(ii)] $\mathcal{A}$ is maximally $\rho-$monotone if $\mathcal{A}$ is $\rho-$monotone and there is no other $\rho-$monotone operator $\mathcal{B}:\mathcal{H}\rightarrow 2^{\mathcal{H}}$ such that $\text{ gra } \mathcal{B}$ properly contains $\text{ gra }\mathcal{A}$.\\ 

 \item[(iii)] $\mathcal{A}$ is $\rho-$comonotone if
$$\langle x-y,u-v\rangle \geq \rho \|u-v\|^2,\quad  \forall (x,u),  (y,v)\in \mbox{ gra }\mathcal{A}.$$
 
\item[(iv)] $\mathcal{A}$ is maximally $\rho-$comonotone if $\mathcal{A}$ is $\rho-$comonotone and there is no other $\rho-$comonotone operator $\mathcal{B}:\mathcal{H}\rightarrow 2^{\mathcal{H}}$ such that $\text{ gra }\mathcal{B}$ properly contains $\text {gra }\mathcal{A}$.
 \end{enumerate}
\end{definition}





\begin{remark} \cite[Remark 2.4]{BauschkeMW}\label{hypo}
\begin{enumerate}
\item[(i)] When $\rho=0$, both $\rho-$monotonicity of $\mathcal{A}$ and $\rho-$comonotonicity of $\mathcal{A}$ reduce to the monotonicity of $\mathcal{A}$; equivalently to the monotonicity of ${\mathcal{A}}^{-1}$.
\item[(ii)] When $\rho<0$, $\rho-$monotonicity is know as $\rho-$hypomonotonicity, see \cite[Example 12.28] {RockafellarWets} and \cite [Definition 6.9.1] {Burachik}. In this case, the $\rho-$comonotonicity is also known as $\rho-$cohypomonotonicity (see \cite[Definition2.2]{CP}).
\item[(iii)] In passing, we point out that when $\rho>0$, $\rho-$monotonicity of  $\mathcal{A}$ reduces to $\rho-$strong monotonicity of $\mathcal{A}$, while $\rho-$comonotonicity of $\mathcal{A}$ reduces to $\rho-$cocoercivity of $\mathcal{A}$.
\end{enumerate}	
\end{remark}
 

\begin{remark}\label{av}  \cite [Remark 2.2] {Tan}
Set $\eta \mathcal{A}=T^{-1}-Id$, $i.e.$, $T=J_{\eta}^{\mathcal{A}}$, and set $\rho=(\frac{1}{2\theta}-1)\eta>-\frac{\eta}{2}$, where $\eta>0$. Then we have: $T$ is $\theta-$averaged $\Leftrightarrow$ $\mathcal{A}$ is $\rho-$comonotone. Furthermore, the crucial relation between $\mathcal{A}$ and $\mathcal{A}_{\eta}$ is that $\mbox{zer} \mathcal{A}=\mbox{zer} \mathcal{A}_{\eta}$, which does not depend on the choice of $\eta> 0$(See \cite[Proposition 23.2] {BauschkeC}  and \cite [Proposition 2.13] {BauschkeMW}).
\end{remark}

\begin{proposition} \cite [Proposition 3.7] {BartzD} \label{single}
Let $\mathcal{A}:\mathcal{H}\rightarrow 2^{\mathcal{H}}$ be $\rho-$comonotone and let $\eta\in \mathbb{R}_{++}$ such that $\rho+\eta>0$. Then
\begin{enumerate}
\item[(i)] $J_{\eta}^{\mathcal{A}}$ is single-valued.
\item[(ii)]  $\mathcal{A}$ is maximally $\rho-$comonotone   if and only if  $\mbox{dom} J_{\eta}^{\mathcal{A}}=\mathcal{H}$.	
\end{enumerate}
\end{proposition}

By Remark $\ref{av}$ and Proposition $\ref{single}$, we can reasonably assume that for $\eta>\max\{-2\rho,0\}$, the resolvent $J_{\eta}^{\mathcal{A}}$ and the Yosida regularization $\mathcal{A}_{\eta}$ are single-valued and averaged when $\mathcal{A}$ is $\rho$-comonotone.

\begin{proposition} \cite[Proposition 2.3] {Tan}\label{z}
Suppose that $\mathcal{A}:\mathcal{H}\rightarrow 2^{\mathcal{H}}$ is a $\rho-$comonotone operator and $\eta>\max\{-2\rho,0\}$. For any $x,y \in \mathcal{H}$, the following conclusions hold:
\begin{enumerate}
\item[(i)]$\mathcal{A}$ is maximally $\rho-$comonotone $\Leftrightarrow$ ${\mathcal{A}}^{-1}-\rho Id$ is maximally monotone.
\item[(ii)]$\mathcal{A}$ is maximally $\rho-$comonotone $\Leftrightarrow$ $\mathcal{A}_{\eta}$ is $(\rho+\eta)-$cocoercive.
\item[(iii)]$J_{\eta}^{\mathcal{A}}: \mathcal{H}\rightarrow \mathcal{H}$ is $\frac{\eta}{2(\rho+\eta)}-$averaged and ${\mathcal{A}}_{\eta}:  \mathcal{H}\rightarrow \mathcal{H}$ is $\frac{1}{\rho+\eta}-$Lipschitz continuous.
\item[(iv)]$\xi\in J_{\eta}^{\mathcal{A}}x \Leftrightarrow (\xi, {\eta}^{-1}(x-\xi))\in \mbox{gra}\mathcal{A}$.
\end{enumerate} 
\end{proposition}

\begin{proposition}\label{der}
Let $\mathcal{A}:\mathcal{H}\rightarrow 2^{\mathcal{H}}$ be maximally $\rho-$comonotone, $\eta>\max\{-2\rho,0\}$, and let $x(t)$ be a differentiable function. Then 
$$\langle \dot{x}(t), \frac{d}{dt}\mathcal{A}_{\eta}x(t)\rangle \geq 0.$$
\end{proposition}
\begin{proof}
In light of $\mathcal{A}_{\eta}x(t)=\frac{x(t)-J_{\eta}^\mathcal{A}x(t)}{\eta}$ and by means of the Cauchy-Schwartz inequality, we have 
\begin{eqnarray*}
&&\langle \dot{x}(t), \frac{d}{dt}\mathcal{A}_{\eta}x(t)\rangle	\\
&=&\frac{1}{\eta}\langle \dot{x}(t), \dot{x}(t)-\frac{d}{dt}J_{\eta}^\mathcal{A}x(t)\rangle\\
&\geq&\frac{1}{\eta}\|\dot{x}(t)\|^2-\frac{1}{\eta}\| \dot{x}(t)\|\cdot \|\frac{d}{dt}J_{\eta}^\mathcal{A}x(t)\|.
\end{eqnarray*}
Since $J_{\eta}^\mathcal{A}$ is nonexpansive (by $(iii)$ of Proposition \ref{z}), we have $\|\frac{d}{dt}J_{\eta}^\mathcal{A}x(t)\|\leq\|\dot{x}(t)\|$.  This together the above inequality yields the desired result.
\end{proof}

\begin{proposition} \cite[Proposition 2.4] {Tan}\label{maximal}
Let $\mathcal{A}:\mathcal{H}\rightarrow 2^{\mathcal{H}}$ be a maximally $\rho-$comonotone operator, $(x_n,u_n)_{n\in \mathbb{N}}$ in $\text{gra}\mathcal{A}$ and  $(x,u)\in \mathcal{H} \times \mathcal{H}$. If $x_n\rightharpoonup x$ and $u_n\longrightarrow u$, then $(x,u)\in \mbox{gra} \mathcal{A}$.
\end{proposition}

\begin{lemma} \cite [Lemma 5.1] {NL} \label{exists}
Suppose that $F:[0,+\infty)\rightarrow \mathbb{R}$ is locally absolutely continuous and bounded below and that there exists $G\in L^{1}([0,+\infty)$ such that for almost every $t\in [0,+\infty)$,
$$\frac{d}{dt}F(t)\leq G(t).$$
Then there exists $\lim_{t\rightarrow +\infty}F(t)\in \mathbb{R}$.	
\end{lemma}


\begin{lemma} \cite [Fact 2.5] {Ou} \label{des}
Let $\{a_n\}_{n\in \mathbb{N}}$ and $\{b_n\}_{n\in \mathbb{N}}$ be sequences in $\mathbb{R}_{+}$ such that $\sum_{n\in \mathbb{N}}b_n< +\infty$ and 
$$(\forall  n\in \mathbb{N}) ~~~~a_{n+1}\leq a_n+b_n.$$	
Then $\lim_{n\rightarrow +\infty}a_n\in \mathbb{R}_{+}$.
\end{lemma}

\begin{lemma}  \cite{Opial} \label{op}
Let $S$ be a nonempty subset of $\mathcal{H}$ and let $\{x_k\}$ be a sequence of elements of $\mathcal{H}$. Assume that
\begin{enumerate}
\item[(i)] for every $z\in S$, $\lim\limits_{k\rightarrow +\infty}\|x_k-z\|$ exists;
\item[(ii)] every weak sequential limit point of $\{x_k\}$, as $k\rightarrow +\infty$, belongs to $S$.
\end{enumerate}
Then $x_k$ converges weakly as $k\rightarrow +\infty$ to a point in $S$.	
\end{lemma}

\section{ Convergence properties of inertial dynamical system}\label{sec3}
In this section  we shall  investigate the convergence properties of \eqref{DS}.  Before this,  we will first give  a short discussion concerning the existence and uniqueness of  a global solution of  \eqref{DS},

\subsection{Existence and uniqueness}

\begin{definition}
We say that $x:[t_0,+\infty)\longrightarrow \mathcal{H}$ is a strong global solution of $(\ref{DS})$ if it satisfies the following properties:
\begin{enumerate}
\item[(i)] $x$, $\dot{x}:[t_0,+\infty)\longrightarrow \mathcal{H}$ are locally absolutely continuous.
\item[(ii)] $\ddot{x}(t)+\frac{\alpha}{t}\dot{x}(t)+\frac{\beta}{t}\mathcal{A}_{\eta}(x(t)+\frac{t}{\gamma}\dot{x}(t))=0$ for almost every $t\geq t_0$.
\end{enumerate}	
\end{definition}


\begin{theorem}
Suppose that  $\mathcal{A}:\mathcal{H}\rightarrow 2^{\mathcal{H}}$ is a maximally $\rho$-comonotone operator  with $\rho\in \mathbb{R}$  and  $\text{ zer }\mathcal{A}\neq \emptyset$ and $\eta>\max\{-2\rho, 0\}$. Then, for any $(u_0,v_0)\in \mathcal{H}\times  \mathcal{H}$, there exists a unique strong global solution $x: [t_0, +\infty) \rightarrow \mathcal{H}$ of the system $(\ref{DS})$ which satisfies the Cauchy data $x(t_0)=u_0$ and $\dot{x}(t_0)=v_0$, where $t_0>0$.
\end{theorem}

\begin{proof}
 Rewrite $(\ref{DS})$ as
\begin{equation*}
\left\{
\begin{array}{lcl}
\dot{x}(t)=y(t),\\
\dot{y}(t)=-\frac{\alpha}{t}y(t)-\frac{\beta}{t}\mathcal{A}_{\eta}\left(x(t)+\frac{t}{\gamma}y(t)\right).
\end{array}
\right.
\end{equation*}
Then, the system $(\ref{DS})$ with the Cauchy data $x(t_0)=u_0$ and $\dot{x}(t_0)=v_0$ is equivalent to  the first order dynamical system
\begin{equation}\label{Z-equ}
\left\{
\begin{array}{lcl}
\dot{Z}(t)=F(t,Z(t)),\\
Z(t_0)=(u_0,v_0),
\end{array}
\right.
\end{equation}	
where  $Z(t):=(x(t),y(t))$ and $$F(t,Z)=(y,-\frac{\alpha}{t}y-\frac{\beta}{t}\mathcal{A}_{\eta}(x+\frac{t}{\gamma}y)),\quad\forall t\ge t_0, \forall Z:=(x,y)\in \mathcal{H}\times\mathcal{H}.$$
By $(iii)$ of Proposition \ref{z}, $\mathcal{A}_{\eta}$ is $\frac{1}{\rho+\eta}$-Lipschitz continuity. According to the Cauchy-Lipschitz-Picard Theorem (see \cite [Proposition 6.2.1] {Haraux}) and using standard arguments,  we know that \eqref{Z-equ} admits a unique strong global solution. The conclusion follows.
\end{proof}

\subsection{Convergence analysis}
\begin{theorem}\label{Th-s}
Assume that  $\mathcal{A}:\mathcal{H}\rightarrow 2^{\mathcal{H}}$ is a maximally $\rho$-comonotone operator with $\text{ zer }\mathcal{A}\neq \emptyset$ and the parameters in the system \eqref{DS}  satisfy  $\eta>\max\{-2\rho, 0\}$, $\alpha>\gamma+2$ and $\beta,\gamma>0$. Let $x:[t_0,+\infty)\rightarrow \mathcal{H}$ be a global strong solution of \eqref{DS}. Then, the following statements are true:
\begin{enumerate}
\item [(i)]  $\int_{t_0}^{+\infty}t\|\dot{x}(t)\|^2dt<+\infty$ and $\int_{t_0}^{+\infty} t\|\mathcal{A}_{\eta}(x(t)+\frac{t}{\gamma}\dot{x}(t))\|^2 dt <+\infty$.
\item [(ii)] $\|\dot{x}(t)\|=o(\frac{1}{t})$ and $\|\mathcal{A}_{\eta}(x(t)+\frac{t}{\gamma}\dot{x}(t))\|=o(\frac{1}{t})$ as $t\rightarrow +\infty$.
\item [(iii)] $x(t)$ converges weakly, as $t\rightarrow +\infty$, to an element of $\mbox{zer} \mathcal{A}$.
\end{enumerate}
\end{theorem}

\begin{proof}
Given $x^*\in \mbox{zer}\mathcal{A}$ and  $b \in [0, \alpha-1]$,  define 
$$\varepsilon_b(t):=\frac{1}{2}\|b(x(t)-x^*)+t\dot{x}(t)\|^2+\frac{b(\alpha-1-b)}{2}\|x(t)-x^*\|^2.$$
Differentiating $\varepsilon_b(t)$ gives
\begin{eqnarray}\label{ee}
\frac{d\varepsilon_b(t)}{dt}&=&\langle b(x(t)-x^*)+t\dot{x}(t),b\dot{x}(t)+\dot{x}(t)+t\ddot{x}(t)\rangle+b(\alpha-1-b)\langle x(t)-x^*,\dot{x}(t)\rangle\nonumber\\
&=&\langle b(x(t)-x^*)+t\dot{x}(t),(b+1-\alpha)\dot{x}(t)-\beta\mathcal{A}_{\eta}(x(t)+\frac{t}{\gamma}\dot{x}(t))\rangle \nonumber\\
&&+b(\alpha-1-b)\langle x(t)-x^*,\dot{x}(t)\rangle\nonumber\\
&=&-b\beta\langle x(t)-x^*, \mathcal{A}_{\eta}(x(t)+\frac{t}{\gamma}\dot{x}(t))\rangle-\beta t\langle \dot{x}(t),\mathcal{A}_{\eta}(x(t)+\frac{t}{\gamma}\dot{x}(t)) \rangle\nonumber\\
&&+(b+1-\alpha)t\|\dot{x}(t)\|^2\nonumber\\
&=&-b\beta  \langle x(t)+\frac{t}{\gamma}\dot{x}(t)-x^*, \mathcal{A}_{\eta}(x(t)+\frac{t}{\gamma}\dot{x}(t))\rangle\nonumber\\
&&+(\frac{b}{\gamma}-1)\beta t\langle \dot{x}(t), \mathcal{A}_{\eta}(x(t)+\frac{t}{\gamma}\dot{x}(t))\rangle+(b+1-\alpha)t\|\dot{x}(t)\|^2\nonumber\\
&\leq&-b\beta(\rho+\eta)\|\mathcal{A}_{\eta}(x(t)+\frac{t}{\gamma}\dot{x}(t))\|^2+(b+1-\alpha)t\|\dot{x}(t)\|^2\nonumber\\
&&+(\frac{b}{\gamma}-1)\beta t\langle \dot{x}(t), \mathcal{A}_{\eta}(x(t)+\frac{t}{\gamma}\dot{x}(t))\rangle,
\end{eqnarray}
where the second equation uses \eqref{DS} and the inequality is from the $(\rho+\eta)$-cocoercivity of $\mathcal{A}_{\eta}$ (by $(ii)$ of Proposition \ref{z}). In particular, taking $b=\gamma$ in $(\ref{ee})$, we have
\begin{equation}\label{tf-27-1}
\frac{d\varepsilon_\gamma(t)}{dt}+\beta\gamma(\rho+\eta)\|\mathcal{A}_{\eta}(x(t)+\frac{t}{\gamma}\dot{x}(t))\|^2+(\alpha-1-\gamma)t\|\dot{x}(t)\|^2\leq0.
\end{equation}
This together with $\alpha>\gamma+2$ and $\eta>\max\{-2\rho, 0\}$ yields  $\frac{d\varepsilon_\gamma(t)}{dt}\leq 0$, which implies that $\varepsilon_\gamma(t)$ is nonincreasing on $[t_0, +\infty)$.
Integrate \eqref{tf-27-1} from $t_0$ to $+\infty$ to get
$$\beta\gamma(\rho+\eta)\int_{t_0}^{+\infty}\|\mathcal{A}_{\eta}(x(t)+\frac{t}{\gamma}\dot{x}(t))\|^2 dt +(\alpha-1-\gamma)\int_{t_0}^{+\infty}t\|\dot{x}(t)\|^2 dt\leq  \varepsilon_\gamma(t_0),$$
which yields
\begin{equation}\label{x}
\int_{t_0}^{+\infty}t\|\dot{x}(t)\|^2 dt< +\infty.	
\end{equation}
Using \eqref{DS}, we have
\begin{eqnarray*}
\frac{d}{dt}(x(t)+\frac{t}{\gamma}\dot{x}(t))&=&\dot{x}(t)+\frac{1}{\gamma}\dot{x}(t)+\frac{t}{\gamma}\ddot{x}(t)\\
&=&\frac{\gamma+1-\alpha}{\gamma}\dot{x}(t)-\frac{\beta}{\gamma}\mathcal{A}_{\eta}(x(t)+\frac{t}{\gamma}\dot{x}(t)),
\end{eqnarray*}
which  together with  Proposition \ref{der} implies 
\begin{eqnarray*}
&&\langle (\alpha-1-\gamma)\dot{x}(t)+\beta\mathcal{A}_{\eta}(x(t)+\frac{t}{\gamma}\dot{x}(t)),\frac{d}{dt}(\mathcal{A}_{\eta}(x(t)+\frac{t}{\gamma}\dot{x}(t)))\rangle\\
&=&	-\gamma \langle \frac{d}{dt}(x(t)+\frac{t}{\gamma}\dot{x}(t)),\frac{d}{dt}(\mathcal{A}_{\eta}(x(t)+\frac{t}{\gamma}\dot{x}(t)))\rangle\\
&\leq&0.
\end{eqnarray*}
This yields
\begin{eqnarray}\label{es}
&&\frac{d}{dt}(t^2\|(\alpha-1-\gamma)\dot{x}(t)+\beta\mathcal{A}_{\eta}(x(t)+\frac{t}{\gamma}\dot{x}(t))\|^2)\nonumber\\
&=&2t\|(\alpha-1-\gamma)\dot{x}+\beta\mathcal{A}_{\eta}(x(t)+\frac{t}{\gamma}\dot{x}(t))\|^2\nonumber\\
&&+2t^2\langle (\alpha-1-\gamma)\dot{x}(t)+\beta\mathcal{A}_{\eta}(x(t)+\frac{t}{\gamma}\dot{x}(t)),\nonumber\\
&&(\alpha-1-\gamma)(-\frac{\alpha}{t}\dot{x}(t)-\frac{\beta}{t}\mathcal{A}_{\eta}(x(t)+\frac{t}{\gamma}\dot{x}(t)))\rangle\nonumber\\
&&+2\beta t^2\langle (\alpha-1-\gamma)\dot{x}(t)+\beta\mathcal{A}_{\eta}(x(t)+\frac{t}{\gamma}\dot{x}(t)),\frac{d}{dt}(\mathcal{A}_{\eta}(x(t)+\frac{t}{\gamma}\dot{x}(t)))\rangle\nonumber\\
&\le&2(2+\gamma-\alpha)t\|(\alpha-1-\gamma)\dot{x}(t)+\beta\mathcal{A}_{\eta}(x(t)+\frac{t}{\gamma}\dot{x}(t))\|^2\nonumber\\
&&-2(\alpha-1-\gamma)(1+\gamma)t \langle (\alpha-1-\gamma)\dot{x}(t)+\beta\mathcal{A}_{\eta}(x(t)+\frac{t}{\gamma}\dot{x}(t)), \dot{x}(t) \rangle.
\end{eqnarray}
Given $\xi>0$, using the Cauchy-Schwarz inequality, we get
\begin{eqnarray*}
&&2\langle (\alpha-1-\gamma)\dot{x}(t)+\beta\mathcal{A}_{\eta}(x(t)+\frac{t}{\gamma}\dot{x}(t)), \dot{x}(t) \rangle\\
&\leq& \xi \|(\alpha-1-\gamma)\dot{x}(t)+\beta\mathcal{A}_{\eta}(x(t)+\frac{t}{\gamma}\dot{x}(t))\|^2+\frac{1}{\xi}\|\dot{x}(t)\|^2,
\end{eqnarray*}
which, combined with \eqref{es}, leads to  
\begin{eqnarray}\label{contra}
&&\frac{d}{dt}(t^2\|(\alpha-1-\gamma)\dot{x}(t)+\beta\mathcal{A}_{\eta}(x(t)+\frac{t}{\gamma}\dot{x}(t))\|^2)\nonumber\\
&\leq&[4+2\gamma-2\alpha+(\alpha-1-\gamma)(1+\gamma)\xi]t\|(\alpha-1-\gamma)\dot{x}(t)+\beta\mathcal{A}_{\eta}(x(t)+\frac{t}{\gamma}\dot{x}(t))\|^2\nonumber\\
&&+(\alpha-1-\gamma)(1+\gamma)\frac{1}{\xi}t\|\dot{x}(t)\|^2.
\end{eqnarray}
Since $\alpha>\gamma+2$ and $\gamma>0$, we can choose $\xi \in (0,\frac{2\alpha-2\gamma-4}{(\alpha-1-\gamma)(1+\gamma)})$ such that
$$4+2\gamma-2\alpha+(\alpha-1-\gamma)(1+\gamma)\xi<0.$$
This together with \eqref{contra} yields
$$\frac{d}{dt}(t^2\|(\alpha-1-\gamma)\dot{x}(t)+\beta\mathcal{A}_{\eta}(x(t)+\frac{t}{\gamma}\dot{x}(t))\|^2)\leq  (\alpha-1-\gamma)(1+\gamma)\frac{1}{\xi}t\|\dot{x}(t)\|^2.$$
In view of $(\ref{x})$, applying Lemma \ref{exists}  to the above inequality, we have
\begin{equation}\label{tf-a}\lim_{t\rightarrow +\infty}t^2\|(\alpha-1-\gamma)\dot{x}(t)+\beta\mathcal{A}_{\eta}(x(t)+\frac{t}{\gamma}\dot{x}(t))\|^2 \text{ exists }.
\end{equation}
Using again $(\ref{x})$ and integrating $(\ref{contra})$ from $t_0$ to $+\infty$, we we obtain
\begin{eqnarray}\label{int}
\int_{t_0}^{+\infty} t\|(\alpha-1-\gamma)\dot{x}(t)+\beta\mathcal{A}_{\eta}(x(t)+\frac{t}{\gamma}\dot{x}(t))\|^2<+\infty.
\end{eqnarray}
Combining \eqref{tf-a} and \eqref{int}, we have
\begin{equation}\label{e0}
\lim_{t\rightarrow +\infty}t^2\|(\alpha-1-\gamma)\dot{x}(t)+\beta\mathcal{A}_{\eta}(x(t)+\frac{t}{\gamma}\dot{x}(t))\|^2=0.
\end{equation}
It follows from $(\ref{x})$ and $(\ref{int})$ that
\begin{eqnarray*}
&&\int_{t_0}^{+\infty} \frac{1}{2}t\|\beta\mathcal{A}_{\eta}(x(t)+\frac{t}{\gamma}\dot{x}(t))\|^2 dt\\
&=&\int_{t_0}^{+\infty} \frac{1}{2}t\|\beta\mathcal{A}_{\eta}(x(t)+\frac{t}{\gamma}\dot{x}(t))+(\alpha-1-\gamma)\dot{x}(t)-(\alpha-1-\gamma)\dot{x}(t)\|^2 dt\\
&\leq& \int_{t_0}^{+\infty}t\|\beta\mathcal{A}_{\eta}(x(t)+\frac{t}{\gamma}\dot{x}(t))+(\alpha-1-\gamma)\dot{x}(t)\|^2 dt+ \int_{t_0}^{+\infty}t\|(\alpha-1-\gamma)\dot{x}(t)\|^2 dt\\
&<& +\infty.
\end{eqnarray*}
This completes the proof of $(i)$. 

Let us go back to $(\ref{ee})$. Neglecting two nonpositive terms, we have
\begin{eqnarray}\label{lem}
\frac{d\varepsilon_b(t)}{dt}&\leq&\beta t(\frac{b}{\gamma}-1)\langle \dot{x}(t),\mathcal{A}_{\eta}(x(t)+\frac{t}{\gamma}\dot{x}(t))\rangle\nonumber\\
&\leq&\beta t|\frac{b}{\gamma}-1|\cdot|\langle \dot{x}(t),\mathcal{A}_{\eta}(x(t)+\frac{t}{\gamma}\dot{x}(t)) \rangle|.
\end{eqnarray}
Using the Cauchy-Schwarz inequality and the statement $(i)$, we have
\begin{eqnarray*}
&&\int_{t_0}^{+\infty} t|\langle \dot{x}(t),\mathcal{A}_{\eta}(x(t)+\frac{t}{\gamma}\dot{x}(t)) \rangle| dt \\
&\leq& \int_{t_0}^{+\infty} \frac{1}{2}t\|\dot{x}(t)\|^2 dt+ \int_{t_0}^{+\infty} \frac{1}{2}t\|\mathcal{A}_{\eta}(x(t)+\frac{t}{\gamma}\dot{x}(t))\|^2 dt \\
&<& +\infty.
\end{eqnarray*}
Now, applying Lemma \ref{exists} to $(\ref{lem})$, we obtain $\lim_{t\rightarrow +\infty}\varepsilon_b(t)$ exists for any $0\leq b\leq \alpha-1$. In particular, taking $b=0$, we have
$\lim_{t\rightarrow +\infty}t^2\|\dot{x}(t)\|^2$ exists. This along with $(\ref{x})$ leads to
\begin{equation}\label{x0}
\lim_{t\rightarrow +\infty}t^2\|\dot{x}(t)\|^2=0.	
\end{equation}
Combining $(\ref{e0})$ and $(\ref{x0})$, we have
\begin{eqnarray*}
&&\lim_{t\rightarrow +\infty}t^2\|\beta\mathcal{A}_{\eta}(x(t)+\frac{t}{\gamma}\dot{x}(t))\|^2\\
&=&\lim_{t\rightarrow +\infty}t^2\|(\alpha-1-\gamma)\dot{x}(t)+\beta \mathcal{A}_{\eta}(x(t)+\frac{t}{\gamma}\dot{x}(t))-(\alpha-1-\gamma)\dot{x}(t)\|^2\\
&\leq &\lim_{t\rightarrow +\infty}2t^2\|(\alpha-1-\gamma)\dot{x}(t)+\beta \mathcal{A}_{\eta}(x(t)+\frac{t}{\gamma}\dot{x}(t))\|^2+\lim_{t\rightarrow +\infty}2t^2\|(\alpha-1-\gamma)\dot{x}(t)\|^2\\
&=& 0.
\end{eqnarray*}
This proves $(ii)$.

Next, we will show $(iii)$ by using the Opial's lemma (Lemma \ref{op}). By the definition of $\varepsilon_b(t)$,
\begin{eqnarray*}
\varepsilon_b(t)&=&\frac{1}{2}b^2\|x(t)-x^*\|^2+\frac{1}{2}t^2\|\dot{x}(t)\|^2+bt\langle x(t)-x^*,\dot{x}(t)\rangle\\
&&+\frac{(\alpha-1-b)b}{2}\|x(t)-x^*\|^2\\
&=&\frac{b(\alpha-1)}{2}\|x(t)-x^*\|^2+\frac{1}{2}t^2\|\dot{x}(t)\|^2+bt\langle  x(t)-x^*, \dot{x}(t)\rangle.	
\end{eqnarray*}
In order to prove the existence of $\lim_{t\rightarrow+\infty}\|x(t)-x^*\|^2$, we only need to prove that $\lim_{t\rightarrow+\infty}t\langle x(t)-x^*, \dot{x}(t)\rangle=0$, because of the existence of $\lim_{t\rightarrow+\infty}\varepsilon_b(t)$ and $(\ref{x0})$. Since $\varepsilon_b(t)$ is bounded on $[t_0,+\infty)$, by definition we know that $x(t)$ is bounded on  $[t_0,+\infty)$. 
Hence, $\lim_{t\rightarrow+\infty}\|x(t)-x^*\|^2$ exists. 

To complete the proof via Lemma \ref{op}, we need to prove that every weak sequential cluster point of $x(t)$ belongs to zer$\mathcal{A}$. Let $\bar{x}$ be a weak sequential cluster point of $x(t)$. Then, there exists  $ t_n\rightarrow +\infty$ such that $x(t_n)\rightharpoonup \bar{x}$ as $n\rightarrow +\infty$. By $(ii)$,  $\mathcal{A}_{\eta}(x(t_n)+\frac{t_n}{\gamma}\dot{x}(t_n))\rightarrow 0$. Passing to the limit in 
$$\mathcal{A}_{\eta}(x(t_n)+\frac{t_n}{\gamma}\dot{x}(t_n))\in \mathcal{A}(x(t_n)-\eta\mathcal{A}_{\eta}(x(t_n)+\frac{t_n}{\gamma}\dot{x}(t_n))),$$
and using Proposition \ref{maximal}, we have $ 0\in \mathcal{A}(\bar{x})$.
Consequently, $x(t)$ converges weakly to an element of zer$\mathcal{A}$.
\end{proof}

\begin{remark}  Here, we compare  the convergence  results of Theorem \ref{Th-s} on \eqref{DS} with the ones reported in \cite[Theorem 12 and Theorem 16]{Alecsa} on \eqref{im}, and \cite[Theorem 4.1]{Tan} on \eqref{ds}.
\begin{itemize}
\item{}\text{ Velocity }: a) The integral estimate $\int_{t_0}^{+\infty}t\|\dot{x}(t)\|^2dt<+\infty$ in Theorem \ref{Th-s}(i) on \eqref{DS} has been derived in  \cite[Theorem 16]{Alecsa} on \eqref{im} and   \cite[Theorem 4.1(i)]{Tan} on \eqref{ds}. b) The point estimate  $\|\dot{x}(t)\|=o(\frac{1}{t})$ in Theorem \ref{Th-s}(ii) on \eqref{DS} is  consistent with the one reported in \cite[Theorem 4.1(i)]{Tan} on \eqref{ds}, and improves the rate   $\|\dot{x}(t)\|=O(\frac{1}{t})$  reported in \cite[Theorem 16]{Alecsa} on \eqref{im}.

\item{}\text{ Yosida regularization  or gradient operator}:  a) The integral estimate 
$$\int_{t_0}^{+\infty} t\|\mathcal{A}_{\eta}(x(t)+\frac{t}{\gamma}\dot{x}(t))\|^2 dt <+\infty$$
in Theorem \ref{Th-s}(i) on \eqref{DS} is consistent with the estimate 
$$\int_{t_0}^{+\infty} t\|\mathcal{A}_{\eta}x(t)\|^2 dt <+\infty$$ reported in  \cite[Theorem 4.1(i)]{Tan} on \eqref{ds}, while  the integral estimate
$$\int_{t_0}^{+\infty}t^2\|\nabla f\left(x(t)+(\gamma+\frac{\beta}{t})\dot{x}(t)\right) \|^2dt<+\infty$$ 
was reported in  \cite[Theorem 12]{Alecsa} on \eqref{im}.

\item{} \text{ Trajectory }: The weak convergence of the trajectory to a solution established  in  Theorem \ref{Th-s}(iii) on \eqref{DS} was  reported in \cite[Theorem 24]{Alecsa} on \eqref{im}  and \cite[Theorem 4.1]{Tan}(iii) on \eqref{ds}.
\end{itemize}
\end{remark}

\section{Convergence analysis of inertial algorithm}\label{sec4}

In this section we will discuss the convergence properties of the algorithm  \eqref{ale}.

\begin{theorem}\label{kehe}
Suppose that $\mathcal{A}:\mathcal{H}\rightarrow 2^{\mathcal{H}}$ is a maximally $\rho$-comonotone operator with  $\text{ zer }\mathcal{A}\neq \emptyset$ and the parameters in \eqref{ale} satisfying  $\eta>\max\{-2\rho, 0\}$, $\alpha>\gamma+2$, $\gamma>\frac{\beta}{2(\rho+\eta)}$ and $\beta>0$. Let $(x_n,y_n,z_n)$ be the  sequence generated by \eqref{ale}. Then the following conclusions hold:
\begin{enumerate}
	\item [(i)] $\sum_{n=1}^{+\infty} n\|x_n-x_{n-1}\|^2<+\infty$,  $\sum_{n=1}^{+\infty} n\|\mathcal{A}_{\eta}(z_{n})\|^2<+\infty$ and 
$$\sum_{n=1}^{+\infty} n^2\|\mathcal{A}_{\eta}(z_{n})-\mathcal{A}_{\eta}(z_{n+1})\|^2<+\infty.$$
	\item [(ii)] $\|x_n-x_{n-1}\|=o(\frac{1}{n})$ and $\|\mathcal{A}_{\eta}(z_n)\|=o(\frac{1}{n})$ as $n\rightarrow +\infty$.
	\item [(iii)]  $x_n$ converges weakly, as $n\rightarrow +\infty$, to an element of $\mbox{zer} \mathcal{A}$. 
\end{enumerate}	
\end{theorem}

\begin{proof}
Given $x^*\in \mbox{zer}\mathcal{A}$ and  $b\in [0, \alpha-1]$, define
\begin{eqnarray*}
\varepsilon_{b,n}:&=&\frac{1}{2}\|b(x_n-x^*)+(n-\alpha)(x_n-x_{n-1})\|^2+\frac{b(\alpha-1-b)}{2}\|x_n-x^*\|^2	\\
&&+\frac{2\alpha-1-b}{2}n\|x_n-x_{n-1}\|^2.
\end{eqnarray*}
It follows that
\begin{eqnarray}\label{eb}
&&\varepsilon_{b,n+1}-\varepsilon_{b,n}\nonumber\\
&=&\frac{1}{2}\|b(x_{n+1}-x^*)+(n+1-\alpha)(x_{n+1}-x_n)\|^2+\frac{b(\alpha-1-b)}{2}\|x_{n+1}-x^*\|^2\nonumber\\
&&-\frac{1}{2}\|b(x_n-x^*)+(n-\alpha)(x_n-x_{n-1})\|^2-\frac{b(\alpha-1-b)}{2}\|x_n-x^*\|^2\nonumber\\
&&+\frac{2\alpha-1-b}{2}(n+1)\|x_{n+1}-x_n\|^2-\frac{2\alpha-1-b}{2}n\|x_n-x_{n-1}\|^2\nonumber\\
&=&b(n+1-\alpha)\langle  x_{n+1}-x^*,x_{n+1}-x_n\rangle-b(n-\alpha)\langle  x_n-x^*,x_n-x_{n-1}\rangle\nonumber\\
&&+\frac{1}{2}[(n-b)(n+1)+{\alpha}^2]\|x_{n+1}-x_n\|^2-\frac{1}{2}[(n-1-b)n+{\alpha}^2]\|x_n-x_{n-1}\|^2\nonumber\\
&&+\frac{b(\alpha-1)}{2}\|x_{n+1}-x^*\|^2-\frac{b(\alpha-1)}{2}\|x_{n}-x^*\|^2.
\end{eqnarray}
From \eqref{ale}, we get
 $$(1-\frac{\alpha}{n})(x_n-x_{n-1})=x_{n+1}-x_n+\frac{\beta}{n}\mathcal{A}_{\eta}(z_n),$$
which yields
\begin{eqnarray*}
&&(n-\alpha)\langle x_n-x_{n-1}, x_{n+1}-x^*\rangle\\
&=&n\langle x_{n+1}-x_n, x_{n+1}-x^*	\rangle+\beta \langle \mathcal{A}_{\eta}(z_n), x_{n+1}-x^*\rangle\\
&=&(n+1-\alpha)\langle x_{n+1}-x_n, x_{n+1}-x^*\rangle +\beta \langle \mathcal{A}_{\eta}(z_n), x_{n+1}-x^*\rangle\\
&&+(\alpha-1) \langle x_{n+1}-x_n, x_{n+1}-x^*\rangle.
\end{eqnarray*}
This implies that
\begin{eqnarray}\label{l3}
&&(n+1-\alpha)\langle x_{n+1}-x_n, x_{n+1}-x^*\rangle-(n-\alpha)\langle x_{n}-x_{n-1}, x_{n}-x^*\rangle\nonumber\\
&=&(n-\alpha)\langle x_{n}-x_{n-1}, x_{n+1}-x_n\rangle-(\alpha-1)\langle x_{n+1}-x_n, x_{n+1}-x^*\rangle\nonumber\\
&&-\beta\langle \mathcal{A}_{\eta}(z_n), x_{n+1}-x^*\rangle\nonumber\\
&=&(n-\alpha)\langle x_{n}-x_{n-1}, x_{n+1}-x_n\rangle-\frac{\alpha-1}{2}\|x_{n+1}-x_n\|^2-\frac{\alpha-1}{2}\|x_{n+1}-x^*\|^2\nonumber\\
&&+\frac{\alpha-1}{2}\|x_{n}-x^*\|^2-\beta\langle \mathcal{A}_{\eta}(z_n), x_{n+1}-x^*\rangle.
\end{eqnarray}
From \eqref{al} and \eqref{ale}, we get
\begin{eqnarray}\label{fta-1}
&&(n-\alpha)\langle x_{n}-x_{n-1}, x_{n+1}-x_{n}\rangle\nonumber\\
&=&(n-\alpha)\langle x_{n}-x_{n-1}, (1-\frac{\alpha}{n})(x_n-x_{n-1})-\frac{\beta}{n}\mathcal{A}_{\eta}(z_n)\rangle\nonumber\\
&=&\frac{1}{n}(n-\alpha)^2\|x_n-x_{n-1}\|^2-\beta\langle (1-\frac{\alpha}{n})(x_n-x_{n-1}), \mathcal{A}_{\eta}(z_n)\rangle
\end{eqnarray}
and 
\begin{eqnarray}\label{fta-2}
&&-\beta\langle \mathcal{A}_{\eta}(z_n), x_{n+1}-x^*\rangle\nonumber\\
&=&-\beta\langle \mathcal{A}_{\eta}(z_n),x_{n}+(1-\frac{\alpha}{n})(x_n-x_{n-1})-\frac{\beta}{n}\mathcal{A}_{\eta}(z_n)-x^*\rangle\nonumber\\
&=&-\beta\langle \mathcal{A}_{\eta}(z_n),x_{n}+\frac{n}{\gamma}(x_n-x_{n-1})-x^* \rangle-\beta\langle \mathcal{A}_{\eta}(z_n),(1-\frac{\alpha}{n}-\frac{n}{\gamma})(x_n-x_{n-1})\rangle\nonumber\\
&&+\frac{{\beta}^2}{n}\|\mathcal{A}_{\eta}(z_n)\|^2\nonumber\\
&=&-\beta\langle \mathcal{A}_{\eta}(z_n),z_n-x^* \rangle-\beta\langle \mathcal{A}_{\eta}(z_n),(1-\frac{\alpha}{n}-\frac{n}{\gamma})(x_n-x_{n-1})\rangle+\frac{{\beta}^2}{n}\|\mathcal{A}_{\eta}(z_n)\|^2\nonumber\\
&\leq& -\beta\langle \mathcal{A}_{\eta}(z_n),(1-\frac{\alpha}{n}-\frac{n}{\gamma})(x_n-x_{n-1})\rangle+[\frac{{\beta}^2}{n}-\beta(\rho+\eta)]\|\mathcal{A}_{\eta}(z_n)\|^2,
\end{eqnarray} 
where the inequality is from the $(\rho+\eta)$-cocoercivity of $\mathcal{A}_{\eta}(z_n)$ (by $(ii)$ of  Proposition \ref{z}).
Substituting \eqref{fta-1} and \eqref{fta-2} into \eqref{l3}, we have
\begin{eqnarray*}
&&(n+1-\alpha)\langle x_{n+1}-x_n, x_{n+1}-x^*\rangle-(n-\alpha)\langle x_n-x_{n-1},x_n-x^*\rangle\\
&\leq&\frac{1}{n}(n-\alpha)^2\|x_n-x_{n-1}\|^2-\beta\langle \mathcal{A}_{\eta}(z_n), (1-\frac{\alpha}{n})(x_n-x_{n-1})\rangle\\
&&-\frac{\alpha-1}{2}\|x_{n+1}-x_n\|^2-\frac{\alpha-1}{2}\|x_{n+1}-x^*\|^2+\frac{\alpha-1}{2}\|x_n-x^*\|^2\\
&&-\beta\langle \mathcal{A}_{\eta}(z_n),(1-\frac{\alpha}{n}-\frac{n}{\gamma})(x_n-x_{n-1})\rangle+[\frac{{\beta}^2}{n}-\beta(\rho+\eta)]\|\mathcal{A}_{\eta}(z_n)\|^2.
\end{eqnarray*}	
This together with  \eqref{eb}) leads to
\begin{eqnarray*}
&&\varepsilon_{b,n+1}-\varepsilon_{b,n}\\
&\leq&\frac{1}{2}[(n-b)(n+1)+{\alpha}^2-b(\alpha-1)]\|x_{n+1}-x_n\|^2\nonumber\\
&&+[\frac{b}{n}(n-\alpha)^2-\frac{(n-1-b)n}{2}-\frac{{\alpha}^2}{2}]\|x_n-x_{n-1}\|^2\\
&&-b\beta(2-\frac{2\alpha}{n}-\frac{n}{\gamma})\langle\mathcal{A}_{\eta}(z_n),x_n-x_{n-1}\rangle+b[\frac{{\beta}^2}{n}-\beta(\rho+\eta)]\|\mathcal{A}_{\eta}(z_n)\|^2.
\end{eqnarray*}
Note that $(n-b)(n+1)+{\alpha}^2-b(\alpha-1)\leq \frac{n^3}{n-\alpha}-2bn$, because of 
\begin{eqnarray*}
&&\frac{n^3}{n-\alpha}-2bn-[(n-b)(n+1)+{\alpha}^2-b(\alpha-1)]\\
&=&(\alpha-1-b)n+\alpha b+\frac{{\alpha}^3}{n-\alpha}\geq 0.
\end{eqnarray*}
Therefore,
\begin{eqnarray}\label{eee}
&&\varepsilon_{b,n+1}-\varepsilon_{b,n}\nonumber\\
&\leq&\frac{1}{2}(\frac{n^3}{n-\alpha}-2bn)\|(1-\frac{\alpha}{n})(x_n-x_{n-1})-\frac{\beta}{n}\mathcal{A}_{\eta}(z_n)\|^2\nonumber\\
&&+[\frac{b}{n}(n-\alpha)^2-\frac{(n-1-b)n}{2}-\frac{{\alpha}^2}{2}]\|x_n-x_{n-1}\|^2\nonumber\\
&&-b\beta(2-\frac{2\alpha}{n}-\frac{n}{\gamma})\langle\mathcal{A}_{\eta}(z_n),x_n-x_{n-1}\rangle+[\frac{{\beta}^2b}{n}-b\beta(\rho+\eta)]\|\mathcal{A}_{\eta}(z_n)\|^2\nonumber\\
&=&-\frac{(\alpha-1-b)n+{\alpha}^2}{2}\|x_n-x_{n-1}\|^2+[\frac{{\beta}^2n}{2(n-\alpha)}-b\beta(\rho+\eta)]\|\mathcal{A}_{\eta}(z_n)\|^2\nonumber\\
&&+(\frac{b}{\gamma}-1)\beta n\langle\mathcal{A}_{\eta}(z_n),x_n-x_{n-1}\rangle.
\end{eqnarray}
Taking $b=\gamma$ in \eqref{eee}, we obtain
 \begin{eqnarray}\label{fta3}
&&\varepsilon_{\gamma,n+1}-\varepsilon_{\gamma,n}\nonumber\\
&\leq& -\frac{(\alpha-1-\gamma)n+{\alpha}^2}{2}\|x_n-x_{n-1}\|^2+[\frac{{\beta}^2n}{2(n-\alpha)}-\gamma\beta(\rho+\eta)]\|\mathcal{A}_{\eta}(z_n)\|^2.
\end{eqnarray}
Since  $\alpha>\gamma+2$ and $\gamma>\frac{\beta}{2(\rho+\eta)}$, it follows from \eqref{fta3} that the nonnegative sequence $\varepsilon_{\gamma,n}$ is nonincreasing when $n$ is large enough. Summing   \eqref{fta3} from $n=1$ to $n=+\infty$, we obtain 
\begin{equation}\label{xhe}
\sum_{n=1}^{+\infty} n\|x_{n}-x_{n-1}\|^2<+\infty.	
\end{equation}
Denote 
\begin{equation}\label{ftaw}
\omega_n:=(\alpha-1-\gamma)(x_n-x_{n-1})+\beta  \mathcal{A}_{\eta}(z_{n}).
\end{equation}
It follows from  \eqref{al} and  \eqref{ale} that
\begin{eqnarray}\label{n1}
&&n^2\|\omega_{n+1}\|^2\nonumber\\
&=&n^2\|(\alpha-1-\gamma)(1-\frac{\alpha}{n})(x_n-x_{n-1})-(\alpha-1-\gamma)\frac{\beta}{n}\mathcal{A}_{\eta}(z_n)+\beta \mathcal{A}_{\eta}(z_{n+1})\|^2\nonumber\\
&=&n^2\|\frac{n-\alpha}{n}\omega_n-[\frac{n\beta-\beta-\gamma\beta}{n}\mathcal{A}_{\eta}(z_n)-\beta \mathcal{A}_{\eta}(z_{n+1})]\|^2\nonumber\\
&=&(n-\alpha)^2\|\omega_n\|^2+n^2\|\frac{n\beta-\beta-\gamma\beta}{n}\mathcal{A}_{\eta}(z_n)-\beta\mathcal{A}_{\eta}(z_{n+1})\|^2\nonumber\\
&&-2n(n-\alpha)\langle \omega_n,\frac{n\beta-\beta-\gamma\beta}{n}\mathcal{A}_{\eta}(z_n)-\beta\mathcal{A}_{\eta}(z_{n+1})\rangle.
\end{eqnarray}
Using \eqref{ale}, we obtain
\begin{eqnarray*}
&&\gamma(z_{n}-z_{n+1})\nonumber\\
&=&\gamma x_{n}+n(x_n-x_{n-1})-\gamma x_{n+1}-(n+1)(x_{n+1}-x_n)\nonumber\\
&=&n(x_n-x_{n-1})-(n+1+\gamma)(x_{n+1}-x_n)\nonumber\\
&=&n(x_n-x_{n-1})-(n+1+\gamma)((1-\frac{\alpha}{n})(x_n-x_{n-1})-\frac{\beta}{n}\mathcal{A}_{\eta}(z_n))\nonumber\\
&=&(\alpha-1-\gamma)(x_n-x_{n-1})+\frac{(1+\gamma)\alpha}{n}(x_n-x_{n-1})+\frac{(n+1+\gamma)\beta}{n}\mathcal{A}_{\eta}(z_n).
\end{eqnarray*}
This together with \eqref{ftaw} implies
\begin{equation}\label{omeg}
\omega_n=\gamma(z_n-z_{n+1})-\frac{(1+\gamma)\alpha}{n}(x_n-x_{n-1})-\frac{(1+\gamma)\beta}{n}\mathcal{A}_{\eta}(z_n).
\end{equation}
It follows that
\begin{eqnarray*}
&&-\langle \omega_n, \frac{n\beta-\beta-\gamma\beta}{n}\mathcal{A}_{\eta}(z_n)-\beta	\mathcal{A}_{\eta}(z_{n+1})\rangle\\
&=&-\langle \omega_n, \beta\mathcal{A}_{\eta}(z_n)-\beta\mathcal{A}_{\eta}(z_{n+1})\rangle+\langle \omega_n, \frac{(1+\gamma)\beta}{n}\mathcal{A}_{\eta}(z_n)\rangle\\
&=&-\langle \gamma(z_n-z_{n-1})-\frac{(1+\gamma)\alpha}{n}(x_n-x_{n-1})-\frac{(1+\gamma)\beta}{n}\mathcal{A}_{\eta}(z_n),\\
&&\beta\mathcal{A}_{\eta}(z_n)-\beta\mathcal{A}_{\eta}(z_{n+1})\rangle+\frac{1+\gamma}{n}\langle \omega_n,\beta\mathcal{A}_{\eta}(z_n)\rangle\\
&=&-\langle \gamma(z_n-z_{n+1}),\beta\mathcal{A}_{\eta}(z_n)-\beta\mathcal{A}_{\eta}(z_{n+1})\rangle\\
&&+\frac{1+\gamma}{n}\langle\alpha(x_n-x_{n-1})+\beta\mathcal{A}_{\eta}(z_n),\beta\mathcal{A}_{\eta}(z_n)-\beta\mathcal{A}_{\eta}(z_{n+1})\rangle\\
&&+\frac{1+\gamma}{n}\langle \omega_n,\omega_n\rangle-\frac{1+\gamma}{n}\langle \omega_n,(\alpha-1-\gamma)(x_n-x_{n-1})\rangle\\
&\leq&-\beta\gamma(\rho+\eta)\|\mathcal{A}_{\eta}(z_n)-\mathcal{A}_{\eta}(z_{n+1})\|^2+\frac{1+\gamma}{n}\langle \omega_n,\beta\mathcal{A}_{\eta}(z_n)-\beta\mathcal{A}_{\eta}(z_{n+1})\rangle\\
&&+\frac{(1+\gamma)^2}{n}\langle x_n-x_{n-1},\beta\mathcal{A}_{\eta}(z_n)-\beta\mathcal{A}_{\eta}(z_{n+1})\rangle\\
&&+\frac{1+\gamma}{n}\|\omega_n\|^2-\frac{1+\gamma}{n}\langle \omega_n,(\alpha-1-\gamma)(x_n-x_{n-1})\rangle
\end{eqnarray*}
where the second equation uses \eqref{omeg} and the inequality is from the $(\rho+\eta)$-cocoercivity of $\mathcal{A}_{\eta}$.
This together with \eqref{n1} leads to
\begin{eqnarray}\label{thf1}
&&n^2\|\omega_{n+1}\|^2-(n-1)^2\|\omega_n\|^2\nonumber\\
&\leq&[(4+2\gamma-2\alpha)n+{\alpha}^2-2\alpha(1+\gamma)-1]\|\omega_n\|^2\nonumber\\
&&+n^2\|\frac{n\beta-\beta-\gamma\beta}{n}\mathcal{A}_{\eta}(z_n)-\beta \mathcal{A}_{\eta}(z_{n+1})\|^2\nonumber\\
&&-2n(n-\alpha)\gamma\beta(\rho+\eta)\|\mathcal{A}_{\eta}(z_n)-\mathcal{A}_{\eta}(z_{n+1})\|^2\nonumber\\
&&+2(n-\alpha)(1+\gamma)\beta\langle \omega_n,\mathcal{A}_{\eta}(z_n)-\mathcal{A}_{\eta}(z_{n+1})\rangle \nonumber\\
&&+2(n-\alpha)(1+\gamma)^2\beta\langle x_n-x_{n-1},\mathcal{A}_{\eta}(z_n)-\mathcal{A}_{\eta}(z_{n+1})\rangle\nonumber \\
&&-2(n-\alpha)(1+\gamma)\langle \omega_n, (\alpha-1-\gamma)(x_n-x_{n-1})\rangle.
\end{eqnarray}
Utilizing the Cauchy-Schwarz inequality, we have for any $\delta_1,\delta_2>0$ 
$$2\langle x_n-x_{n-1},\mathcal{A}_{\eta}(z_n)-\mathcal{A}_{\eta}(z_{n+1})\rangle\leq \|\mathcal{A}_{\eta}(z_n)- \mathcal{A}_{\eta}(z_{n+1})\|^2+\|x_n-x_{n-1}\|^2,$$
$$2\langle \omega_n,\mathcal{A}_{\eta}(z_n)-\mathcal{A}_{\eta}(z_{n+1})\rangle\leq \delta_2\|\omega_n\|^2+\frac{1}{\delta_2}\|\mathcal{A}_{\eta}(z_n)- \mathcal{A}_{\eta}(z_{n+1})\|^2$$
and 
$$2\langle \omega_n, (\alpha-1-\gamma)(x_n-x_{n-1})\rangle\\\leq \delta_1\|\omega_n\|^2+\frac{1}{\delta_1}\|(\alpha-1-\gamma)(x_n-x_{n-1})\|^2.$$
Substituting the last three inequalities into \eqref{thf1}, we obtain
\begin{eqnarray}\label{ehe}
&&n^2\|\omega_{n+1}\|^2-(n-1)^2\|\omega_n\|^2\nonumber\\
&\leq&M_1(n)\|\omega_n\|^2+M_2(n)\|\mathcal{A}_{\eta}(z_n)- \mathcal{A}_{\eta}(z_{n+1})\|^2 +M_3(n)\|x_n-x_{n-1}\|^2, 
\end{eqnarray}
where $$M_1(n):=[(4+2\gamma-2\alpha)+(1+\gamma)(\delta_1+\beta\delta_2)]n+{\alpha}^2-\alpha(1+\gamma)(\delta_1+\beta\delta_2),$$
$$M_2(n):=-2n(n-\alpha)(\rho+\eta)\beta\gamma+{\beta}^2 n^2+(n-\alpha)(1+\gamma)\beta(\frac{1}{\delta_2}+1+\gamma),$$
 and 
 $$M_3(n):=(n-\alpha)(1+\gamma)[\frac{(\alpha-1-\gamma)^2}{\delta_1}+(1+\gamma)\beta].$$
Since  $\alpha>\gamma+2$ and   $\gamma>\frac{\beta}{2(\rho+\eta)}$, we can choose  $\delta_1$ and $\delta_2$ such that $\delta_1+\beta\delta_2<\frac{2\alpha-2\gamma-4}{1+\gamma}$. As a result,
$M_1(n)\leq0$ and $M_2(n)\leq0$ when $n$ is large enough.
By neglecting the first  two nonpositive terms in \eqref{ehe}, we obtain
$$n^2\|\omega_{n+1}\|^2-(n-1)^2\|\omega_n\|^2\leq M_3(n)\|x_n-x_{n-1}\|^2.$$
By using \eqref{xhe} and Lemma \ref{des},  we have $\lim_{n\rightarrow +\infty}n^2\|\omega_n\|^2$ exists.  From $(\ref{xhe})$ and $(\ref{ehe})$, we get 
\begin{equation}\label{5}
\sum_{n=1}^{+\infty} n\|\omega_n\|^2<+\infty	
\end{equation}
and $$\sum_{n=1}^{+\infty} n^2\|\mathcal{A}_{\eta}(z_n)- \mathcal{A}_{\eta}(z_{n+1})\|^2<+\infty.$$
Thanks to the existence of $\lim_{n\rightarrow +\infty}n^2\|\omega_n\|^2$ and $(\ref{5})$, we obtain
\begin{equation}\label{0}
\lim_{n\rightarrow +\infty}n^2\|\omega_n\|^2=0.	
\end{equation}
From $(\ref{xhe})$ and $(\ref{5})$, we derive that
\begin{eqnarray*}
&&\sum_{n=1}^{+\infty} n\|\beta\mathcal{A}_{\eta}(z_n)\|^2\\
&=&\sum_{n=1}^{+\infty} n\|\omega_n-(\alpha-1-\gamma)(x_n-x_{n-1})\|^2\\
&\leq&2\sum_{n=1}^{+\infty} n\|\omega_n\|^2+2\sum_{n=1}^{+\infty} n\|(\alpha-1-\gamma)(x_n-x_{n-1})\|^2\\
&< & +\infty.
\end{eqnarray*}
This, combined with $(\ref{xhe})$, yields
\begin{eqnarray}\label{khe}
&&\sum_{n=1}^{+\infty} n|\langle \mathcal{A}_{\eta}(z_n),x_n-x_{n-1}\rangle|\nonumber\\
&\leq&\sum_{n=1}^{+\infty}\frac{n}{2}\| \mathcal{A}_{\eta}(z_n)\|^2+\sum_{n=1}^{+\infty}\frac{n}{2}\|x_n-x_{n-1}\|^2\nonumber\\
&<&+\infty.	
\end{eqnarray}
The proof of $(i)$ is finished.

Neglecting the nonpositive terms in the right side of $(\ref{eee})$, we have
\begin{eqnarray*}
\varepsilon_{b,n+1}-\varepsilon_{b,n}
&\leq &\beta n(\frac{b}{\gamma}-1)\langle \mathcal{A}_{\eta}(z_n),x_n-x_{n-1}\rangle\\
&\leq& \beta n|\frac{b}{\gamma}-1||\langle \mathcal{A}_{\eta}(z_n),x_n-x_{n-1}\rangle|.	
\end{eqnarray*} 
Applying Lemma $\ref{des}$ to the above inequality with $a_n=\varepsilon_{b,n}$ and $b_n=n|\langle \mathcal{A}_{\eta}(z_n),x_n-x_{n-1}\rangle|$, where $b_n$ is summable because of $(\ref{khe})$, we know $\lim_{n\rightarrow +\infty}\varepsilon_{b,n}$ exists for any $b\in [0,\alpha-1]$. In particular, taking $b=0$, we have $\lim_{n\rightarrow +\infty}\frac{1}{2}(n-\alpha)^2\|x_n-x_{n-1}\|^2$ exists. This together with  $(\ref{xhe})$ yields
\begin{equation}\label{eq0}
\lim_{n\rightarrow +\infty}\frac{1}{2}(n-\alpha)^2\|x_n-x_{n-1}\|^2=0.	
\end{equation}
Furthermore, 
\begin{eqnarray*}
&&n^2\| \beta\mathcal{A}_{\eta}(z_n)\|^2\\
&=&n^2\|\omega_n-(\alpha-1-\gamma)(x_n-x_{n-1})\|^2\\
&\leq&2n^2\|\omega_n\|^2+2n^2\|(\alpha-1-\gamma)(x_n-x_{n-1})\|^2.
\end{eqnarray*}
This together with $(\ref{0})$ and $(\ref{eq0})$ yields
$$\lim_{n\rightarrow +\infty}n^2\|\mathcal{A}_{\eta}(z_n)\|^2=0.$$
This proves  $(ii)$.

Next, we prove $(iii)$ by using the Opial's lemma (Lemma \ref{op}). By definition, 
\begin{eqnarray*}
\varepsilon_{b,n}
&=&\frac{1}{2}(n-\alpha)^2\|x_n-x_{n-1}\|^2+\frac{2\alpha-1-b}{2}n\|x_n-x_{n-1}\|^2\\
&&+b(n-\alpha)\langle x_n-x^*,x_n-x_{n-1}\rangle+\frac{(\alpha-1)b}{2}\|x_n-x^*\|^2\\
&=&\frac{1}{2}[n^2-(1+b)n+{\alpha}^2]\|x_n-x_{n-1}\|^2+b(n-\alpha)\langle x_n-x_{n-1},x_n-x^*\rangle\\
&&+\frac{(\alpha-1)b}{2}\|x_n-x^*\|^2.
\end{eqnarray*}
In order to prove the existence of $\lim_{n\rightarrow +\infty}\|x_n-x^*\|^2$, it is sufficient  to show  
$$(n-\alpha)\langle x_n-x_{n-1},x_n-x^*\rangle\to 0\quad \text{ as }\quad n\to+\infty,$$  
which follows from $(\ref{eq0})$, the  boundedness of $\|x_n-x^*\|$ (guaranteed by the boundedness of $\varepsilon_{b,n}$) and
$$(n-\alpha)|\langle x_n-x_{n-1},x_n-x^*\rangle|\leq (n-\alpha)\|x_n-x_{n-1}\|\cdot\|x_n-x^*\|.$$	
To finish the proof, we need to prove that any weak  cluster point  $\hat{x}$ of $\{x_{n}\}$ belongs to $\text{ zer }\mathcal{A}$. Let $\{x_{n_k}\}$ be a subsequence of $\{x_{n}\}$ such that it  weakly converges to $\hat{x}$ as $k\to+\infty$. Invoking $(\ref{ale})$ and $(\ref{eq0})$, we infer that $z_{n_k}$ weakly converges to $\hat{x}$ as $k\to+\infty$. According to $(ii)$, we have $\mathcal{A}_{\eta}(z_{n_k})\rightarrow 0$ as $k\to+\infty$. Since
$$\mathcal{A}_{\eta} (z_{n_k})\in \mathcal{A}(z_{n_k}-\eta\mathcal{A}_{\eta}(z_{n_k})),$$ 
 by Proposition $\ref{maximal}$ we obtain $0\in \mathcal{A}(\hat{x})$.
\end{proof}

\begin{remark}  Although the algorithms \eqref{ale}  and \eqref{Al-cns},  resulting  respectively from \eqref{DS} and  \eqref{ds} via time discretization,  are different  in formulations, they exhibit similar convergence properties as reported in  Theorem \ref{kehe} and \cite[Theorem 5.1]{Tan}. In Section \ref{sec5}, we will show that  the algorithm \eqref{ale} outperforms over the algorithm \eqref{Al-cns} in numerical experiments.
\end{remark}

\section{Numerical experiments}\label{sec5}

In this section, we illustrate the validity of the  system $(\ref{DS})$ and the algorithm $(\ref{al})$  by two examples. The simulations are conducted in Matlab (version 9.4.0.813654)R2018a. All the numerical procedures are performed on a personal computer with Inter(R) Core(TM) i7-4600U, CORES 2.69GHz and RAM 8.00GB.

\begin{example}\label{Ex1} Consider the following inclusion problem
$$0\in \mathcal{A}(x),$$
 where
$$\mathcal{A}=\begin{pmatrix}
5.7943 & -10.8168 & 10.7544\\
-10.8168 & 65.5739 & -28.2603\\
10.7544 & -28.2603 & 35.6622
\end{pmatrix}$$ 
is a randomly generated semi-positive definite matrix of order $3\times3$. Obviously, it is a maximally monotone operator and it has a single zero $x^*=(0,0,0)$.	
\end{example}
By this example, we illustrate  the theoretical results on  $(\ref{DS})$, compared  with the following implicit Newton-like inertial dynamical system
\begin{equation}\label{b}
\ddot{x}(t)+\alpha \dot{x}(t)+\mathcal{A}(x(t)+\gamma\dot{x}(t))=0,	
\end{equation}
which is just  the system \eqref{B} proposed by Adly et al. \cite{Adly} with
$f\equiv 0$,  $B=\mathcal{A}$, and $\beta_b=\beta_f:=\gamma>0$. 

Figure \ref{figure1} depicts the asymptotical behaviors of  the rescaled iterative errors $\text{ lg }\|x(t)-x^*\|$, $\text{ lg } \|\dot{x}(t)\|$, $\text{ lg }\|\mathcal{A}_{\eta}(x(t)+\frac{t}{\gamma}\dot{x}(t))\|$ of \eqref{DS} with  $\eta=2$ and $\text{ lg }\|\mathcal{A}(x(t)+\gamma\dot{x}(t))\|$ of \eqref{b}  under different choices of the parameters $\alpha, \beta, \gamma$. Figure \ref{figure11} depicts the asymptotical behaviors of the trajectory $x(t)$ generated by \eqref{DS} and \eqref{b} with  $\alpha=15$ and $\beta=\gamma=10$. Here,  all  the dynamical systems are solved numerically  by the ode45 function in Matlab on the interval $[0.1,100]$ with  a same initial value  $x(t_0)=(1,-10,-20)$ and $\dot{x}(t_0)=(1,1,1)$. As shown in Figure \ref{figure1} and Figure \ref{figure11}, the system \eqref{DS} outperforms  over \eqref{b}.

\begin{example}\label{Ex2} Consider the following inclusion problem
$$0\in \mathcal{A}(x),$$
 where $\mathcal{A}=
\begin{pmatrix}
-\frac{2}{5} & \frac{4}{5}\\
-\frac{4}{5} & -\frac{2}{5}
\end{pmatrix}$
is a  maximally $\rho$-comonotone operator with $\rho= -\frac{1}{2}$. It is easily verified that $\mathcal{A}$ has a single zero $x^*=(0,0)$.	
\end{example}
 
This example is taken from \cite[Example 6.2]{Tan},  which was used by Tan et al. \cite{Tan}  to test the system \eqref{ds} and its resulting algorithm \eqref{Al-cns} for solving the maximally comonotone inclusion problem. Here, we will use it to test the system \eqref{DS} and the algorithm \eqref{al}.

Figure \ref{figure2} depicts the asymptotical behavior of the trajectory $x(t)$,  the velocity $\dot{x}(t)$ and the Yosida regularization $\mathcal{A}_{\eta}(x(t)+\frac{t}{\gamma}\dot{x}(t))$ generated by $\eqref{DS}$   with parameters $\alpha=5$, $\beta=2$, $\gamma=2.5$ and $\eta=2$, where it is solved numerically by  the ode45 function in Matlab on the interval $[0.1,100]$ with the initial value $x(t_0)=(10,-10)$ and $\dot{x}(t_0)=(1,1)$. In Figure \ref{figure2}, the blue curve represents the first component of $x(t)$, $\dot{x}(t)$ and $\mathcal{A}_{\eta}(x(t)+\frac{t}{\gamma}\dot{x}(t))$ and the red represents the second component of these. It supports the theoretical results of Theorem \ref{Th-s}.

Next, we test the algorithm \eqref{al} on Example \ref{Ex2}, compared with the Halpern-type proximal point algorithm (HPPA) in \cite{Kohlenbach}, the Optimized Halpern's Method (OHM) in \cite{Cai},  and  the algorithm \eqref{Al-cns} proposed  by Tan et al. \cite{Tan}.

Figure \ref{figure3} displays the profiles of $\text{ lg }\|x_{k}-x^*\|$ for the sequences $x_n$ generated by HPPA with $\alpha_n=\frac{1}{n+1}$, $\gamma_n\equiv 2$ and $\mu=(1,1)$, OHM in \cite{Cai} with $\beta_n=\frac{1}{n+1}$ and $\omega_0=(1,1)$, the algorithm \eqref{Al-cns} with  $\alpha=10$, $\beta=4$ and $\eta=2$,   and the algorithm \eqref{al} with $\alpha=10$, $\beta=4$, $\gamma=7$ and $\eta=2$,  under the same stopping criteria $\|x_k-x^*\|\leq 10^{-7}$ with a same value $x_1=x_0=(1,1)$.  The numerical results support the theoretical findings of Theorem \ref{kehe} and shows that  the algorithm \eqref{al} performs better than other algorithms.

\begin{figure*}[h]
 \centering
 {
  \begin{minipage}[t]{0.31\textwidth}
   \centering
   \includegraphics[width=1.8in]{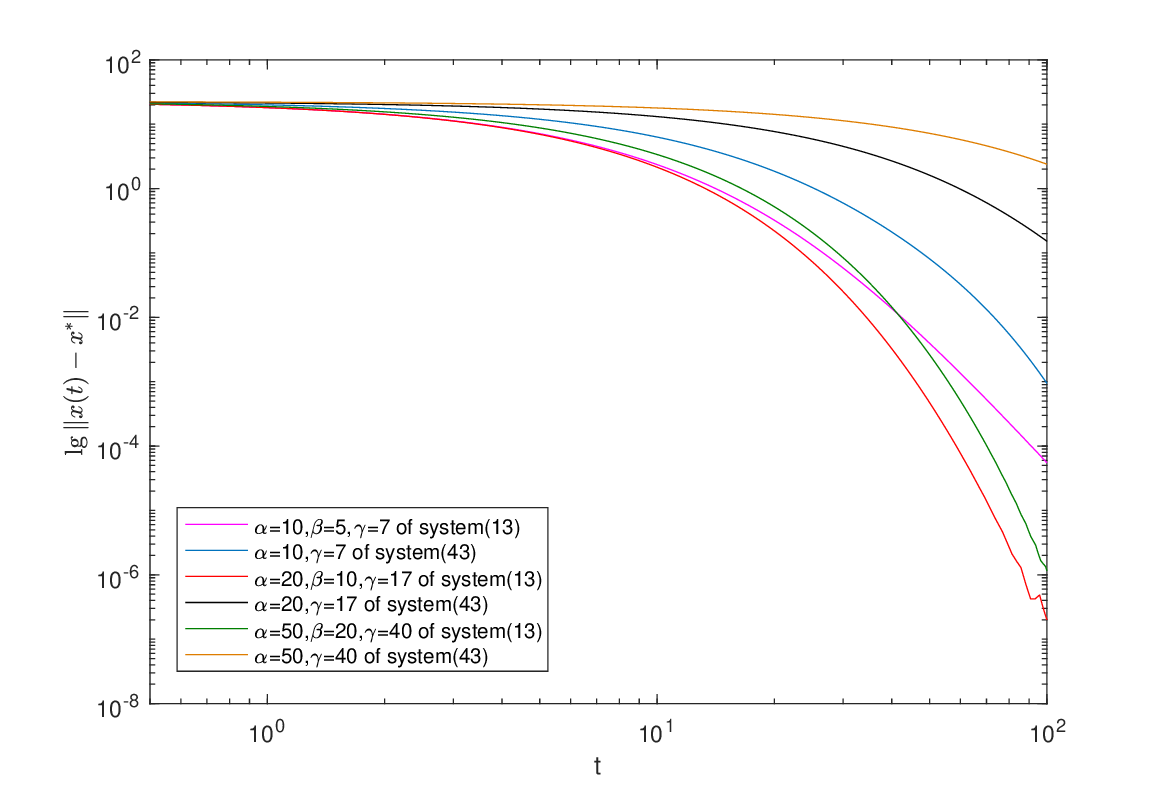}
  \end{minipage}
 }
 {
  \begin{minipage}[t]{0.31\textwidth}
   \centering
   \includegraphics[width=1.8in]{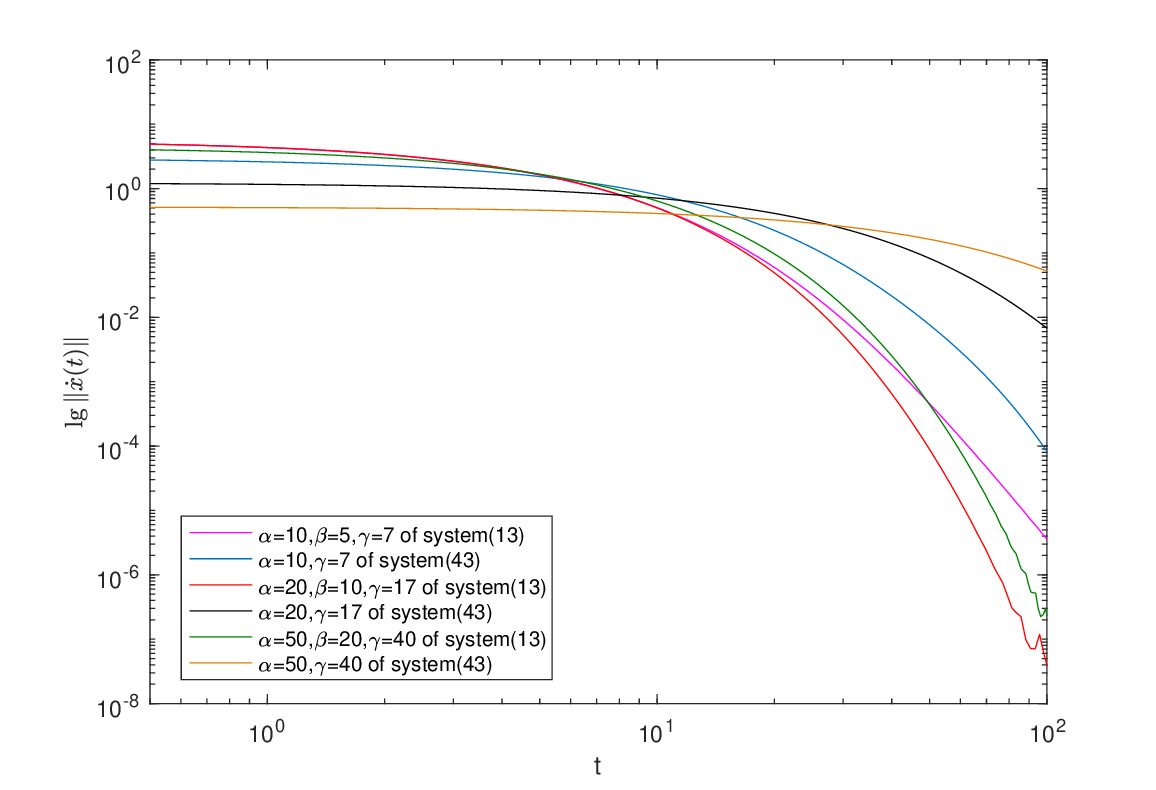}
  \end{minipage}
 }
 {
  \begin{minipage}[t]{0.31\textwidth}
   \centering
   \includegraphics[width=1.8in]{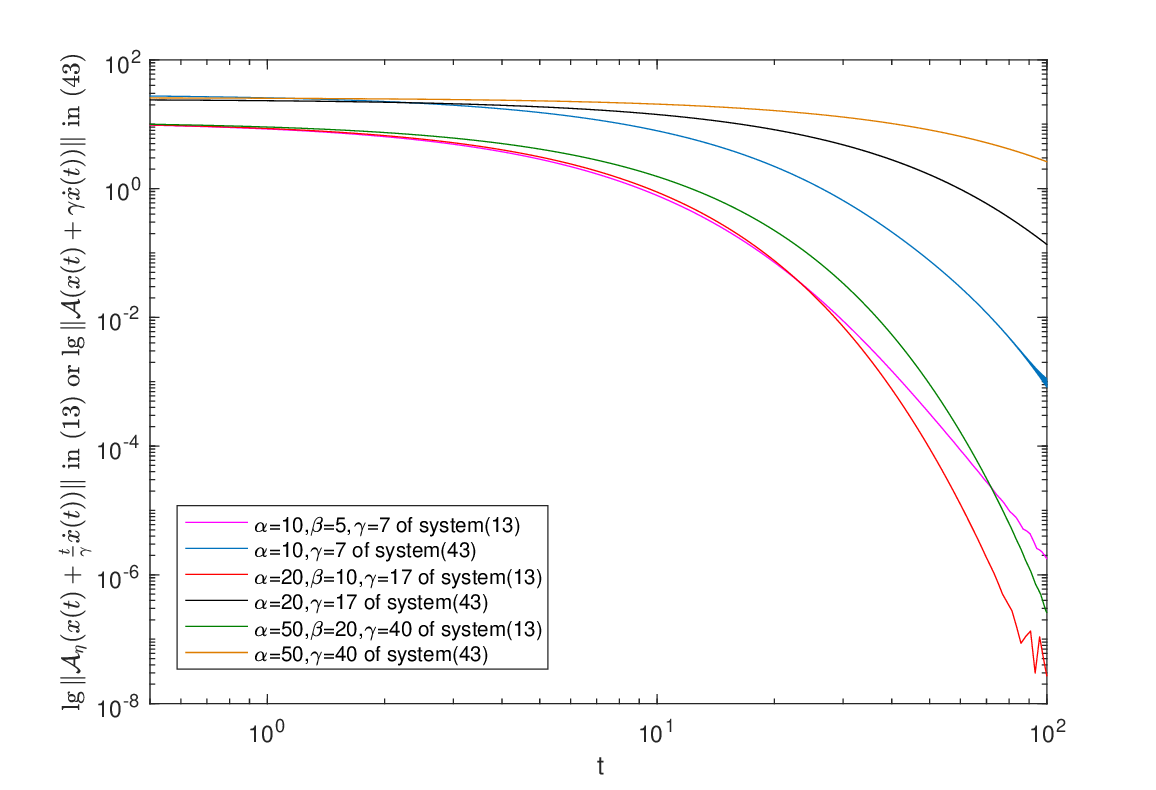}
  \end{minipage}
 }
 \caption{Rescaled iteration errors $\text{ lg } \|x(t)-x^*\|$, $\text{ lg } \|\dot{x}(t)\|$, $\text{ lg }\|\mathcal{A}_{\eta}(x(t)+\frac{t}{\gamma}\dot{x}(t))\|$ of \eqref{DS} and $\text{ lg }\|\mathcal{A}(x(t)+\gamma\dot{x}(t))\|$ of \eqref{b} in Example \ref{Ex1}}
 \label{figure1}
 \centering
\end{figure*}

\begin{figure}[htp]
\begin{center}
\includegraphics[width=0.6\textwidth]{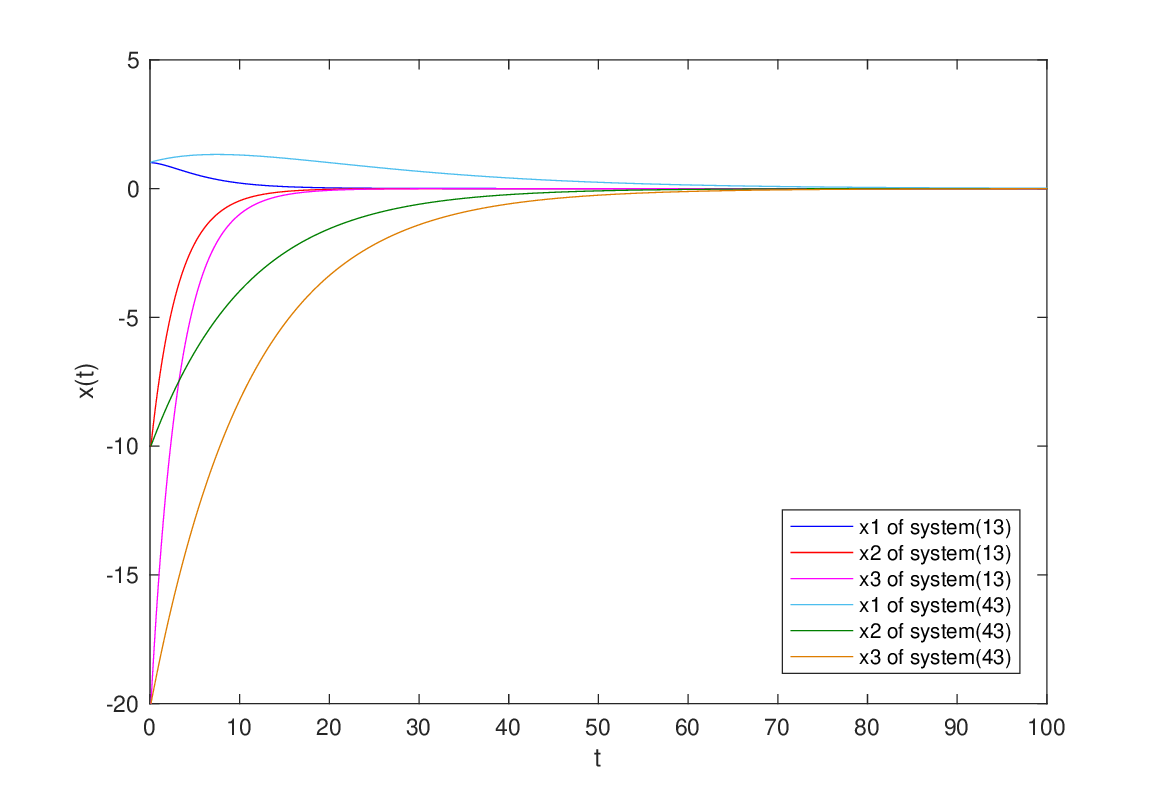}
\caption{Transient behaviors of  the trajectories $x(t)$ for \eqref{DS} and \eqref{b} in Example \ref{Ex1}}
\label{figure11}
\end{center}
\end{figure}

\begin{figure*}[h]
 \centering
 {
  \begin{minipage}[t]{0.31\linewidth}
   \centering
   \includegraphics[width=1.8in]{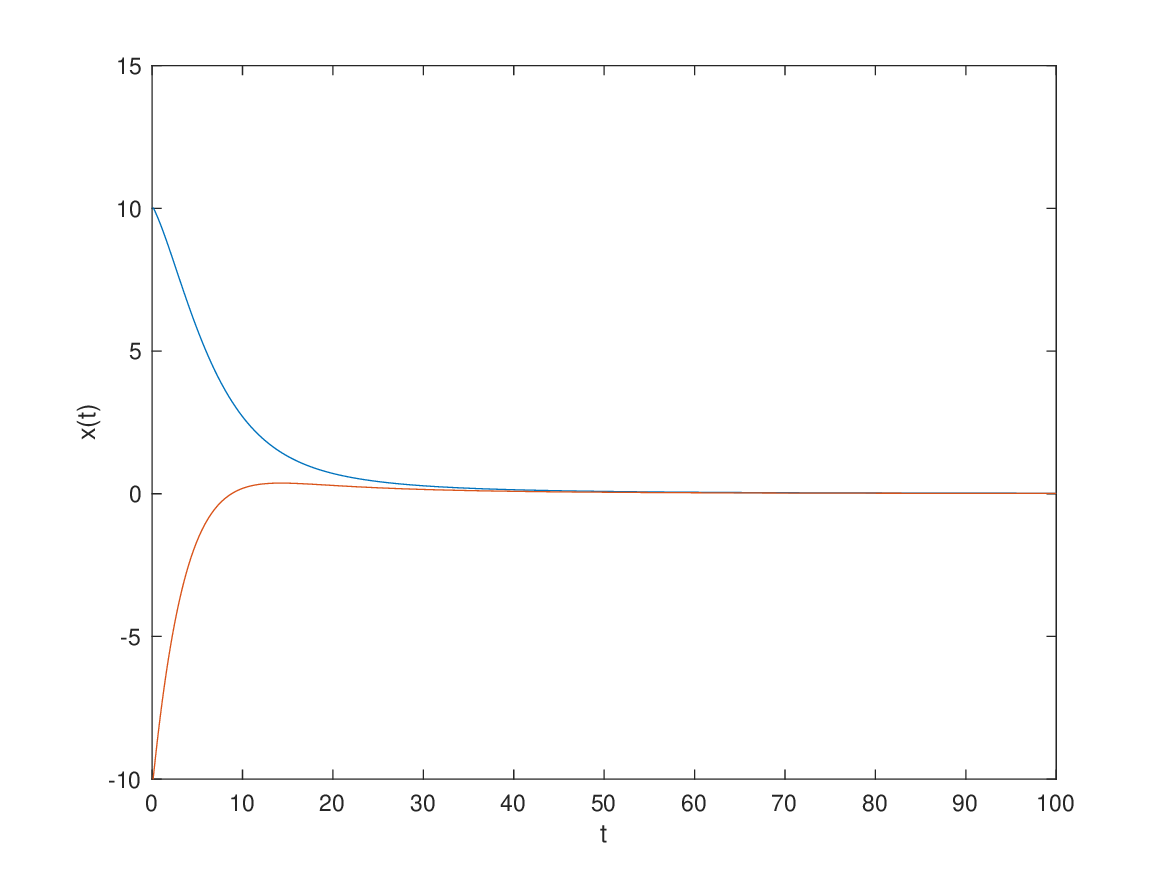}
  \end{minipage}
 }
 {
  \begin{minipage}[t]{0.31\linewidth}
   \centering
   \includegraphics[width=1.8in]{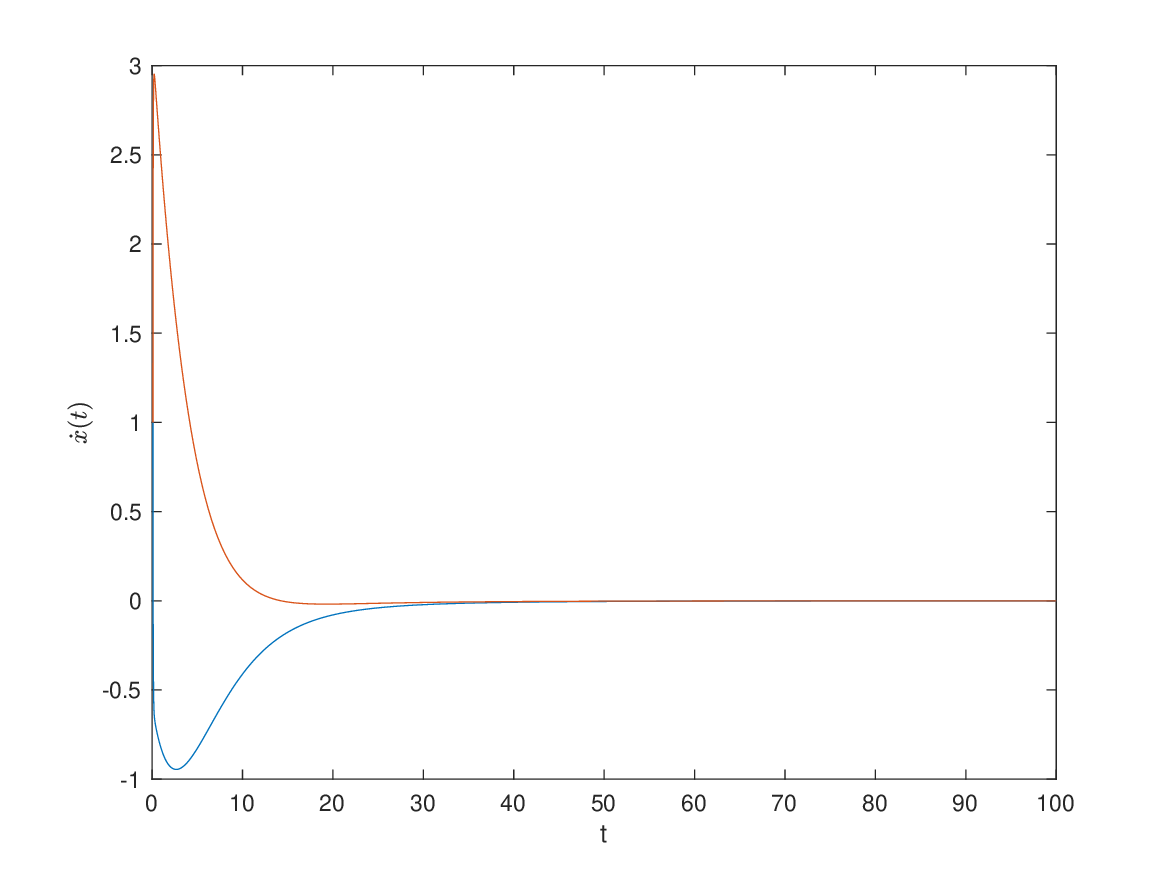}
  \end{minipage}
 }
 {
  \begin{minipage}[t]{0.31\linewidth}
   \centering
   \includegraphics[width=1.8in]{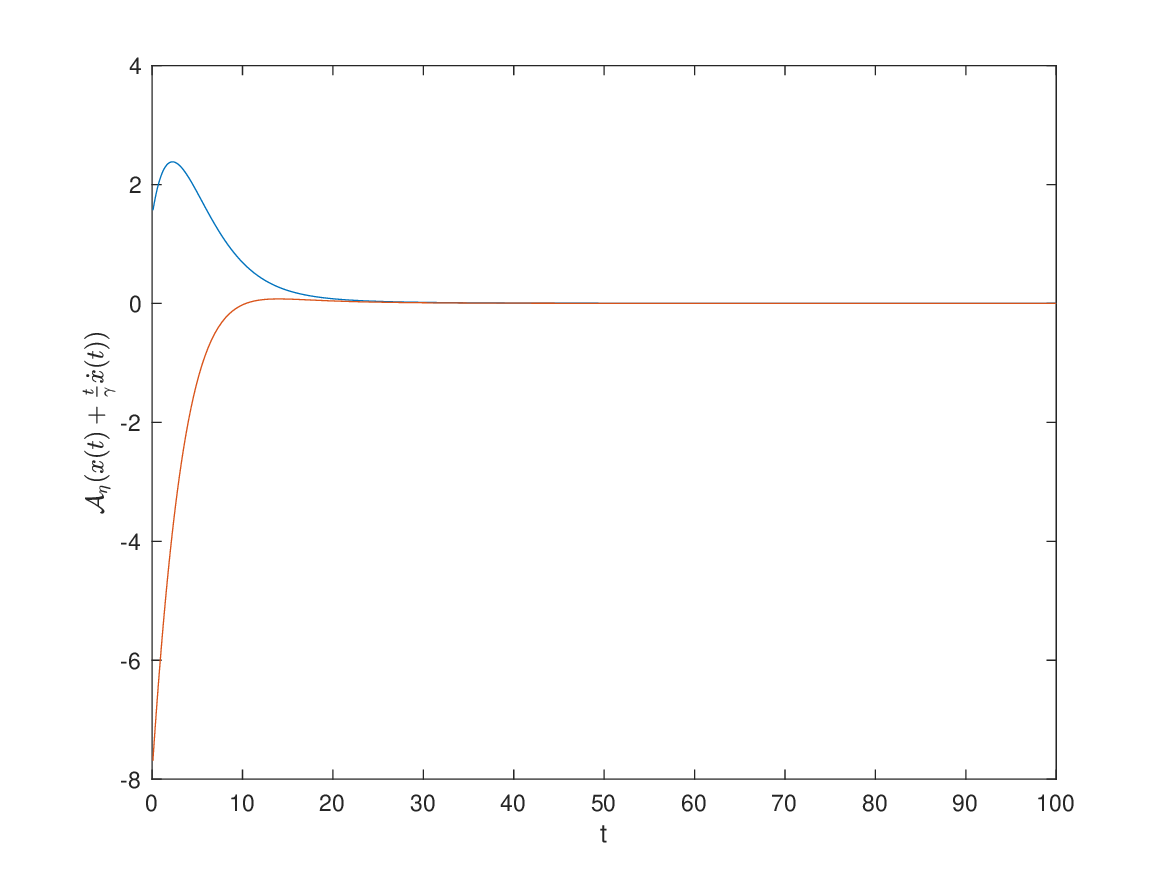}
  \end{minipage}
 }
 \caption{Transient behavior of the trajectory $x(t)$, $\dot{x}(t)$, $\mathcal{A}_{\eta}(x(t)+\frac{t}{\gamma}\dot{x}(t))$ of \eqref{DS} in Example \ref{Ex2}}
 \label{figure2}
 \centering
\end{figure*}

\begin{figure}[htp]
\begin{center}
\includegraphics[width=0.6\textwidth]{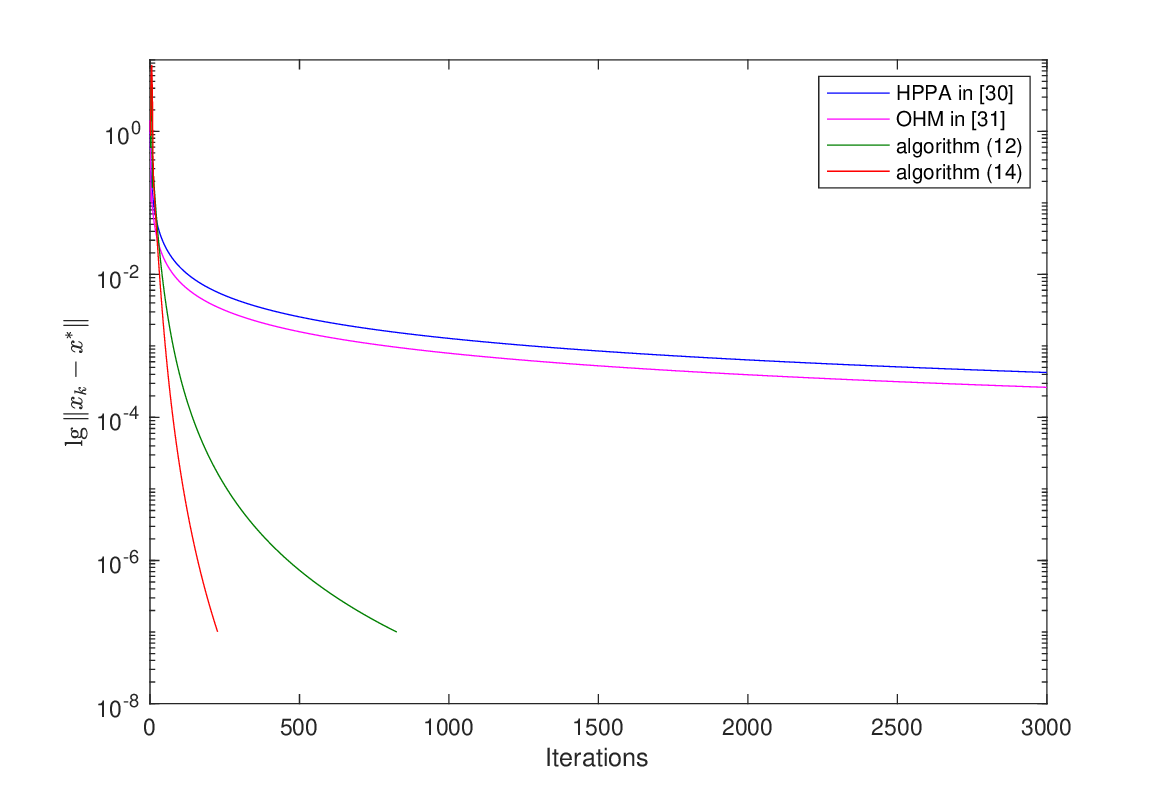}
\caption{Rescaled iteration errors $\text{ lg } \|x_n-x^*\|$ for \eqref{al}, \eqref{Al-cns}, HPPA  and OHM in Example \ref{Ex2}}
\label{figure3}
\end{center}
\end{figure}

\section{Conclusion}\label{sec6}

In this paper we propose an implicit Newton-like inertial dynamical system  governed by a maximally comonotone operator, which  is similar to  the inertial dynamical system \eqref{ds} recently considered by Tan et al. \cite{Tan} but includes implicitly the Newton-like correction term, for solving a maximally comonotone inclusion problem. We show that the proposed system exhibits fast convergence properties, in terms of pointwise estimates  and integral estimates  for the velocity and the value of  the associated Yosida regularization operator along the trajectory, and  the weak convergence of the trajectory to a zero of the underlying operator. By discretizing  our system, we develop a new inertial algorithm for finding a zero of the maximally comonotone operator. Contrast to the algorithm \eqref{Al-cns} due to Tan et al. \cite{Tan}, the algorithm under consideration enjoys similar theoretical convergence properties, but better performance in numerical experiments.


\section*{Declarations}

{\bf Conflict of interest} No potential conflict of interest was reported by the authors.


\bibliography{sn-bibliography}

\end{document}